\documentclass[12pt,a4paper]{amsart}
\calclayout
\numberwithin{equation}{section}
\allowdisplaybreaks
\theoremstyle{plain}

\newtheorem{thrm}{Theorem}[section]
\newtheorem{lemma}[thrm]{Lemma}
\newtheorem{prop}[thrm]{Proposition}

\newtheorem{remark}{Remark}

\numberwithin{equation}{section}

\usepackage{enumerate}

\usepackage{amssymb}
\usepackage{mathtools}
\mathtoolsset{showonlyrefs}
\usepackage{mathrsfs}
\usepackage{comment}
\usepackage{hyperref}
\usepackage{xcolor}

	\def\MR#1{\href{https://urldefense.com/v3/__http://www.ams.org/mathscinet-getitem?mr=*1*7D*7BMR-*1__;IyUlIw!!DZ3fjg!oHzqzx-fk_p2DpzqdMBj54OAPH9BYzIUMFry2bFVial4aXytl4As9A_7US17LHEg$ }}
	\def\ARXIV#1{\href{https://urldefense.com/v3/__https://arxiv.org/abs/*1*7D*7BarXiv:*1__;IyUlIw!!DZ3fjg!oHzqzx-fk_p2DpzqdMBj54OAPH9BYzIUMFry2bFVial4aXytl4As9A_7UeWlCDxg$ }}
	
\def\DOI#1{\href{https://urldefense.com/v3/__https://doi.org/*1*7D*7Bdoi:*1__;IyUlIw!!DZ3fjg!oHzqzx-fk_p2DpzqdMBj54OAPH9BYzIUMFry2bFVial4aXytl4As9A_7UVT3CNkl }}

\def\P{\mathbb{P} }

\def\R{\mathbb{R} }
\def\N{\mathbb{N} }

\def\E{\mathbb{E} }

\def\F{\mathcal{F}}

\begin{document}
\title[Branching Brownian motion in a periodic environment]
{Branching Brownian motion in a periodic environment and
existence of pulsating travelling waves}
\thanks{The research of this project is supported by the National Key R\&D Program of China (No. 2020YFA0712900).}
\author[Y.-X. Ren, R. Song and F. Yang]{Yan-Xia Ren, Renming Song and Fan Yang}
\address{Yan-Xia Ren\\ LMAM School of Mathematical Sciences \& Center for
Statistical Science\\ Peking University\\ Beijing 100871\\ P. R. China}
\email{yxren@math.pku.edu.cn}
\thanks{The research of Y.-X. Ren is supported in part by NSFC (Grant Nos. 12071011  and 11731009) and LMEQF
}
\address{Renming Song\\ Department of Mathematics\\ University of Illinois at Urbana-Champaign \\ Urbana \\ IL 61801\\ USA}
\email{rsong@illinois.edu}
\thanks{The research of R. Song is supported in part by a grant from the Simons Foundation (\#429343, Renming Song)}
\address{Fan Yang\\ School of Mathematical Sciences \\ Peking University\\ Beijing 100871\\ P. R. China}
\email{fan-yang@pku.edu.cn}
\begin{abstract}
We study the limits of the additive and derivative martingales of one-dimensional branching Brownian motion in a periodic environment.
Then we prove the existence of pulsating travelling wave solutions of the corresponding F-KPP equation in the supercritical and critical cases by representing the solutions probabilistically in terms of the limits of  the additive and derivative martingales.
We also prove that there is no pulsating travelling wave solution in the subcritical case.
Our main tools are the spine decomposition and martingale change of measures.
\end{abstract}
\maketitle


\section{Introduction}
\subsection{Background}\label{sec:BGD}
A classical branching Brownian motion (BBM) in $\R$ is constructed as follows. Initially there is a single particle at the origin of the real line. This particle moves as a standard Brownian motion $B = \{B(t), t\geq 0 \}$ and produces a random number of offspring,
$1+L$, after an exponential time $\eta$. We assume that $L$ has distribution $\{p_k, k\geq 0\}$ with $m:=\sum_{k\geq 0} kp_k<\infty$ and $\eta$ is exponentially distributed with parameter $\beta>0$. Starting from their points of creation, each of these children evolves independently.
	
McKean \cite{Mc} established the connection between BBM and
the Fisher-Kolmogorov-Petrovskii-Piskounov (F-KPP)
reaction-diffusion equation
\begin{equation}\label{KPPeq1}
\frac{\partial \mathbf u}{\partial t} = \frac{1}{2} \frac{\partial^2 \mathbf u}{\partial x^2} + \beta({\mathbf f}({\mathbf u})-{\mathbf u}),
\end{equation}
where ${\mathbf f}(s)={\mathbf E}(s^{L+1})$ and ${\mathbf u}:\R^+ \times \R \rightarrow [0,1]$.
The F-KPP equation has been studied intensively by both analytic techniques (see, for example, Kolmogorov et al. \cite{KPP} and Fisher \cite{Fish}) and probabilistic methods
(see, for instance, McKean \cite{Mc}, Bramson \cite{Bramson78,Bramson83}, Harris \cite{Harris99} and Kyprianou \cite{Ky}).
	
Particular attention has been paid to solutions of the form
${\mathbf u}(t,x) = \Phi_c(x-ct).$
Substituting this into \eqref{KPPeq1} shows that $\Phi_c$ satisfies
\begin{equation}\label{travel}
\frac{1}{2} \Phi_c'' + c\Phi_c' +
\beta({\mathbf f}(\Phi_c)-\Phi_c) = 0,
\end{equation}
and such a solution $\Phi_c$ is known as a travelling wave solution of speed $c$.
Kyprianou \cite{Ky}, using
the additive and derivative martingales of BBM,
gave a probabilistic representation of travelling wave solutions and also gave probabilistic proofs, different from Harris \cite{Harris99}, for the existence, asymptotics and uniqueness of travelling wave solutions. Inspired by \cite{Ky}, we consider similar problems for BBMs in a periodic environment.
	
BBM in a periodic environment  (BBMPE) is constructed in the same way as BBM,
except that the constant branching rate is replaced by
a space-dependent rate function $\mathbf g$, where we assume $\mathbf g\in C^1(\R)$ is strictly positive and 1-periodic.
More precisely, initially there is a single particle $v$ at
$x\in \R$, performing standard Brownian motion, and
let $b_v$ and $d_v$ be the birth time and death time of the particle $v$, respectively, and $X_v(s)$ be the location of the particle $v$ at time $s$, then
\begin{equation}\label{KPP-P}
\P_x
\left(d_v-b_v>t \;|\; b_v, \{X_v(s): s\geq b_v \} \right) =
\exp\left\{ -\int_{b_v}^{b_v+t} \mathbf g(X_v(s)) \mathrm{d}s \right\}.
\end{equation}
Here, we fictitiously extend $X_v(s)$ beyond its death time $d_v$.
We still assume that $1+L$
is the number of offspring produced by one individual and $L$ has distribution $\{p_k: k\geq 0 \}$ with $m=\sum_{k\geq 0} kp_k<\infty$.
We always assume that $m>0$.
The F-KPP equation related to BBMPE has the following form:
\begin{equation} \label{KPPeq2}
\frac{\partial \mathbf u}{\partial t} = \frac{1}{2} \frac{\partial^2 \mathbf u}{\partial x^2} +
\mathbf g\cdot(\mathbf f(\mathbf u)-\mathbf u).
\end{equation}
In this case, travelling wave solutions, that is solutions satisfying \eqref{travel}, do not exist. However, we can consider so-called pulsating travelling waves, that is, solutions $\mathbf u:\R^+\times\R\rightarrow [0,1]$ to \eqref{KPPeq2} satisfying
\begin{equation}\label{pul_travel}
\mathbf u(t+\frac{1}{\nu}, x) = \mathbf u(t,x-1),
\end{equation}
as well as the boundary condition
\begin{equation}
\lim_{x\rightarrow-\infty} \mathbf u(t,x) = 0,\quad \lim_{x\rightarrow+\infty} \mathbf u(t,x) =1,
\end{equation}
when $\nu>0$, and
\begin{equation}
\lim_{x\rightarrow-\infty} \mathbf u(t,x) = 1,\quad \lim_{x\rightarrow+\infty} \mathbf u(t,x) =0,
\end{equation}
when $\nu<0$. $\nu$ is called the wave speed.
It is known that there is a constant $\nu^*>0$ such that
when $|\nu|<\nu^*$ no such solution exists, whereas for each $|\nu|\geq \nu^*$
there exists a unique, up to time-shift, pulsating travelling wave (see Hamel et al. \cite{HNRR}).

Recently, Lubetzky, Thornett and Zeitouni \cite{LTZ} established the connection
between F-KPP equation \eqref{KPPeq2} with BBMPE in the case of binary branching
and studied the the maximum of BBMPE.
In this paper we first study limits of additive and derivative martingales of BBMPE.
Then we use these limits to give
probabilistic representations of pulsating travelling wave solutions of \eqref{KPPeq2} with speed
$\nu$ satisfying $|\nu|\geq\nu^*$, where $\nu^*$ is a constant defined below.
In the rest of this paper,
$|\nu|>\nu^*$ is called the supercritical case, and $|\nu|=\nu^*$  the critical case.
We also prove that there is no  pulsating travelling wave solution of \eqref{KPPeq2} with speed satisfying $|\nu|<\nu^*$ (called  the subcritical case).
The asymptotic behavior and  uniqueness of the pulsating travelling wave solution of \eqref{KPPeq2} will be studied in
the companion paper \cite{RSYb}.
Therefore, we extend the results of Kyprianou \cite{Ky} for classical BBM to  BBMPE.
It turns out that most of the general deas in Kyprianou \cite{Ky} still work for BBMPE. However, as we will see, carrying out the actual argument is much more difficult.

Before we state our main results, we first introduce the minimal speed $\nu^*$.
As in \cite{HNRR}, for every $\lambda\in\R$, let $\gamma(\lambda)$ and $\psi(\cdot,\lambda)$ be the principal eigenvalue and the corresponding positive eigenfunction of the periodic problem: for all $x\in \R$,
	\begin{equation}\label{eigen}
	\begin{split}
	\frac{1}{2} \psi_{xx}(x, \lambda) - \lambda \psi_x(x, \lambda) + (\frac{1}{2}\lambda^2
	+m\mathbf g(x))\psi(x, \lambda) &= \gamma(\lambda)\psi(x, \lambda),\\
	\psi(x+1,\lambda) &= \psi(x,\lambda).
	\end{split}
	\end{equation}
We normalize $\psi(\cdot,\lambda)$ such that $\int_0^1 \psi(x,\lambda) dx = 1$. Define
	\begin{equation}\label{crit_v}
	\nu^*:=\min_{\lambda>0} \frac{\gamma(\lambda)}{\lambda},
	\quad \lambda^*:= \underset{\lambda>0}{\arg\min} \frac{\gamma(\lambda)}{\lambda}.
	\end{equation}
$\nu^*$ is  the minimal wave speed (see \cite{HNRR}) and the existence of $\lambda^*$ is proved in \cite{LTZ}.
	
Let $N_t$ be the set of particles alive at time $t$ and $X_u(s)$ be the position of the particle $u$ or its ancestor at time $s$ for any $u\in N_t$, $s\leq t$. Define
$$Z_t = \sum_{u\in N_t} \delta_{X_u(t)},$$
and $\F_t = \sigma(Z_s: s\leq t)$.
$\{Z_t: t\ge 0\}$ is called a branching Brownain motion in a periodic environment
(BBMPE).
Let $\P_x$ be the law of $\{Z_t: t\ge 0\}$ when the initial particle starts at $x\in\R$, that is $\P_x(Z_0=\delta_x) = 1$ and $\E_x$ be expectation with respect to $\P_x$.
For simplicity, $\P_0$ and $\E_0$ will be written as $\P$ and $\E$, respectively. Notice that the distribution of $L$ does not depend on the spatial location.
In the remainder of this paper, expectations with respect to $L$ will be written as
$\mathbf E$.
\subsection{Main results}
	For any $\lambda\in\R$, define
	\begin{equation}\label{mart_add}
	W_t(\lambda) = e^{-\gamma(\lambda)t} \sum_{u\in N_t} e^{-\lambda X_u(t)} \psi(X_u(t),\lambda).
	\end{equation}

\begin{thrm}\label{thrm1}
For any $\lambda\in\R$ and $x\in\R$, $\{(W_t(\lambda))_{t\geq 0}, \P_x\}$ is a martingale.
The limit $W(\lambda,x) := \lim_{t\uparrow\infty} W_t(\lambda)$ exists $\P_x$-almost surely.

$\mathrm{(i)}$ If $|\lambda| > \lambda^*$ then $W(\lambda,x) = 0$ $\P_x$-almost surely.

$\mathrm{(ii)}$ If $|\lambda| = \lambda^*$ then $W(\lambda,x) = 0$ $\P_x$-almost surely.

$\mathrm{(iii)}$ If $|\lambda| < \lambda^*$ then $W(\lambda,x) = 0$ $\P_x$-almost surely when $\mathbf E(L\log^+L) = \infty$, and $W(\lambda,x)$ is an $L^1(\P_x)$-limit when $\mathbf E(L\log^+L) < \infty$.
\end{thrm}
	
$\{(W_t(\lambda))_{t\geq 0}, \P_x\}$
is called the additive martingale of the BBMPE starting from $x$.
	
\begin{remark}
For BBM, there is no essential difference between the case when the initial ancestor starting from $x$ and the case starting from the origin, whereas things are different for BBMPE because the branching rate depends on the position. This is why we write $W(\lambda,x)$ as the almost sure limit of $W_t(\lambda)$, instead of $W(\lambda)$. Nevertheless, it can be proved that $(W(\lambda,y), \P_y) \overset{d}{=}( e^{-\lambda(y-x)} W(\lambda,x), \P_x)$, if $y-x\in\mathbb{Z}$, due to the 1-periodicity of $\mathbf g$.
\end{remark}
	
We will show that both  $\gamma(\lambda)$ and $\psi(x,\lambda)$ are differentiable with respect to $\lambda$, so we can define
\begin{equation}\label{mart_deriv}
\partial W_t(\lambda) := -\frac{\partial}{\partial \lambda} W_t(\lambda) = e^{-\gamma(\lambda)t} \sum_{u\in N_t} e^{-\lambda X_u(t)} \bigg{(} \psi(X_u(t),\lambda) (\gamma'(\lambda)t+X_u(t)) - \psi_{\lambda}(X_u(t),\lambda)  \bigg{)}.
\end{equation}

\begin{thrm}\label{thrm2}
For any $\lambda\in\R$ and $x\in\R$,
$\{(\partial W_t(\lambda))_{t\geq 0}, \P_x\}$ is a martingale.
For all $|\lambda|\geq \lambda^*$, the limit $\partial W(\lambda,x) := \lim_{t\uparrow\infty} \partial W_t(\lambda)$ exists $\P_x$-almost surely.

$\mathrm{(i)}$ If $|\lambda| > \lambda^*$ then $\partial W(\lambda,x) = 0$ $\P_x$-almost surely.

$\mathrm{(ii)}$ If $|\lambda| = \lambda^*$ then $\partial W(\lambda,x) = 0$ $\P_x$-almost surely when
$\mathbf E(L(\log^+L)^{2}) = \infty$, and $\partial W(\lambda,x)\in (0,\infty)$
(respectively $\partial W(\lambda,x)\in (-\infty,0)$) $\P_x$-almost surely when $\lambda>0$  (respectively $\lambda<0$) and $\mathbf E(L(\log^+L)^{2}) < \infty$.
\end{thrm}

$\{(\partial W_t(\lambda))_{t\geq 0}, \P_x\}$
is called the derivative martingale of the BBMPE starting from $x$.

Pulsating travelling waves have been studied analytically in many papers,
see, for example, \cite{BH, H08, HR, HNRR}.
It is known that there is no pulsating travelling waves for $|\nu|<\nu^*$, whereas if $|\nu|\geq \nu^*$ then there exists a unique, up to time-shift, pulsating travelling wave (see \cite[P. 467]{HNRR}).
In this paper, we use limits of the additive and derivative martingales to give probabilistic representations of pulsating travelling wave in the supercritical case of  $|\nu|> \nu^*$ and the critical case of $|\nu|=\nu^*$.
Thus we give probabilistic proofs of the existence of  pulsating travelling waves in the supercritical and critical cases.
We also give a probabilistic proof of the non-existence of  pulsating travelling waves in the subcritical case $|\nu|<\nu^*$.
In \cite{RSYb}, we will give a probabilistic proof of the asymptotic behavior and uniqueness of the pulsating travelling waves.

\begin{thrm}\label{thrm3}
$\mathrm{(i)}$  \textbf{Supercriticality case.}
If $|\nu|>\nu^*$ and $\mathbf E(L\log^+L) < \infty$,
\begin{equation}
\mathbf u(t,x)
= \E_x \left(\exp\left\{ -e^{\gamma(\lambda)t} W(\lambda,x) \right\} \right)
\end{equation}
is a pulsating travelling wave with speed $\nu$,
where $|\lambda| \in (0,\lambda^*)$ is such that
$\nu = \frac{\gamma(\lambda)}{\lambda}$.

$\mathrm{(ii)}$ \textbf{Criticality case.}
If $|\nu|=\nu^*$ and $\mathbf E(L(\log^+L)^{2}) < \infty$,
\begin{equation}
\mathbf u(t,x) =
\E_x\left(\exp\left\{ -e^{\gamma(\lambda^*)t} \partial W(\lambda^*,x) \right\} \right)
\end{equation}
is a pulsating travelling wave with speed $\nu^*$, and
$$\mathbf u(t,x) = \E_x\left(\exp\left\{ -e^{\gamma(\lambda^*)t} \partial W(-\lambda^*,x) \right\} \right)$$
is a pulsating travelling wave with speed $-\nu^*$.

$\mathrm{(iii)}$ \textbf{Subcriticality case.}
There is no pulsating travelling wave when $|\nu|<\nu^*$.
\end{thrm}
	
Since $\gamma(\lambda)$ is an even function (see Lemma \ref{lemma_gamma} below),
in the remainder of the paper we will deal only with the case that $\lambda\geq 0$ (i.e., $\nu\geq 0$) unless otherwise stated. The case $\lambda<0$ (i.e., $\nu<0$) follows from symmetry and the results are the same.

\section{Preliminaries}

\subsection{Properties of principal eigenvalue and eigenfunction}

 We first discuss some properties of $\gamma(\lambda)$,

\begin{lemma}\label{lemma_gamma}
(1) The function $\gamma$ is
analytic and strictly convex on $\R$.
There exists a unique $\lambda^*>0$ such that
\begin{equation*}
\nu^*=\frac{\gamma(\lambda^*)}{\lambda^*} = \min_{\lambda>0} \frac{\gamma(\lambda)}{\lambda}>0.
\end{equation*}
Furthermore
\begin{equation}\label{gamma'}
\lim_{\lambda\rightarrow -\infty}\gamma'(\lambda) = -\infty, \quad \lim_{\lambda\rightarrow +\infty}\gamma'(\lambda) = +\infty.
\end{equation}

(2) The function $\gamma$ is even on $\R$.
\end{lemma}
\begin{proof}
The analyticity, convexity of $\gamma$, and the existence and uniqueness of $\lambda^*$ are contained in \cite[Lemma 2.5]{LTZ}.
For the analyticity and convexity of $\gamma$, one can also see \cite[Lemma 2.1]{H08}.
Now we show \eqref{gamma'}.
Define $\phi(x,\lambda) = e^{-\lambda x} \psi(x,\lambda)$, for $\lambda\in\R$, $x\in\R$. A direct calculation shows that $\phi$ satisfies
\begin{equation}
\frac{1}{2}\phi_{xx}(x,\lambda) + m\mathbf g(x)\phi(x,\lambda) = \gamma(\lambda) \phi(x,\lambda).
\end{equation}
By the Feynman-Kac formula,
$$
\phi(x,\lambda)= \Pi_x \left[\phi(B_t,\lambda) e^{-\gamma(\lambda)t + m\int_0^t\mathbf g(B_s)\mathrm{d}s}   \right],\quad x\in\R.
$$
Hence,
\begin{equation}\label{psi_feynman_Kac}
\psi(x,\lambda)= \Pi_x \left[\psi(B_t,\lambda) e^{-\gamma(\lambda)t-\lambda (B_t-x) + m\int_0^t\mathbf g(B_s)\mathrm{d}s}   \right],\quad x\in\R.
\end{equation}
Since $\mathbf g$ is 1-periodic and continuous, we can assume that $0<\alpha\leq \mathbf g(x)\leq \beta<\infty$ for all $x\in\R$.
Notice that   $\Pi_x[e^{-\lambda B_t}]=e^{-\lambda x+\frac{\lambda^2 t}{2}}$, \eqref{psi_feynman_Kac} implies that
\begin{equation}\label{gamma_range}
\gamma(\lambda) \in \bigg{[} \frac{\lambda^2}{2}+m\alpha,  \frac{\lambda^2}{2}+m\beta \bigg{]}.
\end{equation}
Combining this with the analyticity and convexity of $\gamma$, we get \eqref{gamma'}.

If  there exists $\lambda_1 < \lambda_2$ such that $\gamma'(\lambda_1) = \gamma'(\lambda_2)$. The convexity of $\gamma$ would imply that $\gamma'$ is constant on $[\lambda_1,\lambda_2]$, and then analyticity would imply that $\gamma'$ is a constant, which contradicts \eqref{gamma'}. Therefore, $\gamma'(\cdot)$ is strictly increasing on $\R$, that is, $\gamma(\cdot)$ is strictly convex.

(2) Let $\psi(x): = \psi(x,\lambda)$ satisfy \eqref{eigen}, i.e., $\psi$
is the positive eigenfunction corresponding to the eigenvalue $\gamma(\lambda)$.
Let $\bar{\psi} = \psi(x,-\lambda)$
be the positive eigenfunction corresponding to the eigenvalue $\gamma(-\lambda)$, then $\bar{\psi}(x)$
satisfies, for all $x\in \R$,
\begin{align}
\frac{1}{2} \bar{\psi}_{xx}(x) + \lambda \bar{\psi}_x(x) + (\frac{1}{2}\lambda^2+m\mathbf g(x))\bar{\psi}(x) &= \gamma(-\lambda) \bar{\psi}(x),\label{phi}\\
\bar{\psi}(x+1) &= \bar{\psi}(x).
\end{align}
Multiplying \eqref{phi} by $\psi$ and integrating over $(0,1)$, we get that
\begin{align*}
\gamma(-\lambda) \int_0^1 \bar{\psi}\psi \mathrm{d}x =& \int_0^1 \left(\frac{1}{2} \bar{\psi}_{xx}\psi + \lambda \bar{\psi}_x\psi + (\frac{1}{2}\lambda^2+m\mathbf g(x))\bar{\psi}\psi \right) \mathrm{d}x\\
=& \frac{1}{2} \bar{\psi}_{x}(1)\psi(1) - \frac{1}{2} \bar{\psi}_{x}(0)\psi(0) + \lambda \bar{\psi}(1) \psi(1) - \lambda \bar{\psi}(0) \psi(0)\\
&\, + \int_0^1 \left(-\frac{1}{2} \bar{\psi}_x\psi_x - \lambda \bar{\psi}\psi_x + (\frac{1}{2}\lambda^2+m\mathbf g(x))\bar{\psi}\psi \right) \mathrm{d}x\\
=& \int_0^1 \left(-\frac{1}{2} \bar{\psi}_x\psi_x - \lambda \bar{\psi}\psi_x + (\frac{1}{2}\lambda^2+m\mathbf g(x))\bar{\psi}\psi \right) \mathrm{d}x\\
=& \int_0^1 \left(\frac{1}{2} \bar{\psi}\psi_{xx} - \lambda \bar{\psi}\psi_x + (\frac{1}{2}\lambda^2+m\mathbf g(x))\bar{\psi}\psi \right) \mathrm{d}x\\
=& \gamma(\lambda) \int_0^1 \bar{\psi}\psi \mathrm{d}x.
\end{align*}
Since $\psi,\bar{\psi}>0$, we obtain that $\gamma(\lambda)=\gamma(-\lambda)$.
\end{proof}

We compare the values of $\gamma'(\lambda)$ and $\frac{\gamma(\lambda)}{\lambda}$ in the following lemma.
	
\begin{lemma}\label{lemma_com}
(1) $\gamma'(\lambda^*) = \dfrac{\gamma(\lambda^*)}{\lambda^*}$.
(2) If $0 < \lambda < \lambda^*$, $\gamma'(\lambda) < \dfrac{\gamma(\lambda)}{\lambda}$.
(3) If $\lambda>\lambda^*$,  $\gamma'(\lambda) > \dfrac{\gamma(\lambda)}{\lambda}$.
\end{lemma}
\begin{proof} Put $f(\lambda) = \frac{\gamma(\lambda)}{\lambda}$.
(1) Note that
\begin{equation}
f'(\lambda) = \frac{\gamma'(\lambda)-\frac{\gamma(\lambda)}{\lambda}}{\lambda}.
\end{equation}
Since $f(\lambda^*) = \min_{\lambda>0} f(\lambda)$, we have $f'(\lambda^*)=0$, that is, $\gamma'(\lambda^*) = \frac{\gamma(\lambda^*)}{\lambda^*}$.
	
(2) If there were $0 < \lambda_1 < \lambda^*$ satisfying $f'(\lambda_1)\geq 0$, then, by the uniqueness of $\lambda^*$,
\begin{equation*}
\gamma'(\lambda_1) \geq \frac{\gamma(\lambda_1)}{\lambda_1} > \frac{\gamma(\lambda^*)}{\lambda^*} = \gamma'(\lambda^*).
\end{equation*}
This contradicts the convexity of $\gamma$.
	
(3) If there were $\lambda_2>\lambda^*$ satisfying $f'(\lambda_2) < 0$, let $\lambda_3 = \sup\{\lambda: \lambda < \lambda_2 \text{ and } f'(\lambda)\geq 0 \}$.  Then by the continuity of $f'$, $f'(\lambda_3)=0$ and $\lambda_3<\lambda_2$.
By the definition of $\lambda_3$,
\begin{equation*}
\gamma'(\lambda_3) = \frac{\gamma(\lambda_3)}{\lambda_3} = f(\lambda_3) > \frac{\gamma(\lambda_2)}{\lambda_2} > \gamma'(\lambda_2),
\end{equation*}
which contradicts the convexity of $\gamma$ again.
	
Suppose there were $\lambda_4>\lambda^*$ satisfying $f'(\lambda_4) = 0$.
For any $\delta\in (0,\lambda_4-\lambda^*)$, we have
\begin{equation}
\int_{\lambda_4-\delta}^{\lambda_4} \frac{\gamma'(\lambda)-\frac{\gamma(\lambda)}{\lambda}}{\lambda} \mathrm{d}\lambda = \int_{\lambda_4-\delta}^{\lambda_4} f'(\lambda) \mathrm{d}\lambda = f(\lambda_4) - f(\lambda_4-\delta).
\end{equation}
We claim that $f(\lambda_4)-f(\lambda_4-\delta)>0$. In fact, since $f'(\lambda)\geq 0$ for $\lambda\in [\lambda_4-\delta,\lambda_4]$, we have $f(\lambda_4)-f(\lambda_4-\delta)\geq 0$. If $f(\lambda_4) - f(\lambda_4-\delta) = 0$,
then $f'(\lambda) = 0$ for $\lambda\in [\lambda_4-\delta,\lambda_4]$. Thus we have
\begin{equation*}
\gamma'(\lambda_4) = \frac{\gamma(\lambda_4)}{\lambda_4} = f(\lambda_4) = f(\lambda_4-\delta) = \frac{\gamma(\lambda_4-\delta)}{\lambda_4-\delta} =
\gamma'(\lambda_4-\delta),
\end{equation*}
which contradicts
the strict convexity of $\gamma$. Thus the claim is true.

For $\lambda\in [\lambda_4-\delta,\lambda_4]$, it holds that $\gamma'(\lambda) \leq \gamma'(\lambda_4) = f(\lambda_4)$, and that $f(\lambda) \geq f(\lambda_4-\delta)$ because we have proved $f'(\lambda)\geq 0$ when $\lambda>\lambda^*$. Thus, we obtain
\begin{align*}
f(\lambda_4) - f(\lambda_4-\delta) &= \int_{\lambda_4-\delta}^{\lambda_4} \frac{\gamma'(\lambda)-\frac{\gamma(\lambda)}{\lambda}}{\lambda} \mathrm{d}\lambda
\leq \delta \frac{f(\lambda_4)-f(\lambda_4-\delta)}{\lambda^*}>0,
\end{align*}
which implies that $\lambda^* \leq \delta$ for any $\delta>0$.
This contradicts the fact that $\lambda^*>0$. The proof of (3) is complete.
\end{proof}

\begin{lemma}\label{tildepsi}
Suppose $\widetilde\psi(\cdot,\lambda)$ is a positive eigenfunction of
\eqref{eigen} with  $\widetilde\psi(0,\lambda)=1$. Then
for any $x\in\R$, $\widetilde\psi(x,\cdot)\in  C(\R) \cap C^1(\R\setminus\{0\})$.
Moreover, $\widetilde\psi_\lambda(x,\lambda)$, the derivative of  $\widetilde\psi$ with respective to $\lambda$, satisfies, for all $x\in \R$,
\begin{align}\label{diff-lambda-tildepsi}
&\frac{1}{2} \widetilde\psi_{\lambda xx}(x, \lambda) - \widetilde\psi_x(x, \lambda) - \lambda \widetilde\psi_{\lambda x}(x, \lambda) + \left(\frac{1}{2}\lambda^2 + m\mathbf g(x)\right)\widetilde\psi_{\lambda}(x, \lambda) + \lambda\widetilde\psi(x, \lambda)\\
&= \gamma(\lambda) \widetilde\psi_{\lambda}(x, \lambda) + \gamma'(\lambda) \widetilde\psi(x, \lambda). \nonumber
\end{align}
\end{lemma}
\begin{proof} Since $\widetilde\psi(x,\lambda)$ is 1-periodic,
to prove that, for any $x\in\R$, $\widetilde\psi(x,\cdot)\in  C(\R) \cap C^1(\R\setminus\{0\})$,
 it suffices to verify that, for any $x\in[0, 1]$,
  $\widetilde\psi(x,\lambda)$ is continuous with respect to $\lambda$ and
continuously differentiable with respect to $\lambda$ on $\R\setminus\{0\}$.
Define
	\begin{equation}\label{def_f}
   \widehat  \psi(x,\lambda):=
	e^{-\lambda x} \widetilde\psi(x,\lambda).
	\end{equation}
Then $\widehat  \psi$ satisfies
	\begin{equation}\label{f_equ}
	\frac{1}{2} \widehat  \psi_{xx}(x,\lambda) + m\mathbf g(x) \widehat  \psi(x,\lambda) = \gamma(\lambda) \widehat  \psi(x,\lambda), \quad x\in\R.
	\end{equation}
Therefore, $\widehat  \psi$ is the solution of the following boundary value problem
\begin{equation}
\left\{\begin{array}{rl}
&\frac{1}{2} \widehat  \psi_{xx}(x,\lambda) + m\mathbf g(x)\widehat  \psi(x,\lambda) = \gamma(\lambda) \widehat  \psi(x,\lambda), \quad x\in (0,1),  \\
&\widehat  \psi(1,\lambda) = e^{-\lambda},\quad \widehat  \psi(0,\lambda)=1.
\end{array}\right.
\end{equation}
By the probabilistic representation of the solution to the above boundary value problem, we have
\begin{equation}\label{def-f}
\widehat  \psi(x,\lambda) =  \Pi_x \left[e^{-\lambda B_{\tau}} e^{\int_0^{\tau} (m\mathbf g(B_t)-\gamma(\lambda)) \mathrm{d}t} \right], \quad x\in[0,1],
\end{equation}
where $\tau = \inf\{t > 0: B_t \notin (0,1) \}$.
By \eqref{def_f}, we only need to prove that for any $x\in\R$, $\widehat  \psi(x,\cdot)\in  C(\R) \cap C^1(\R\setminus\{0\})$.

Note that for any $x\in[0,1]$,
$$\widehat  \psi(x,\lambda) = \Pi_x \left[ e^{\int_0^{\tau} (m\mathbf g(B_t)-\gamma(\lambda)) \mathrm{d}t}, B_\tau=0 \right]+e^{-\lambda}\Pi_x \left[ e^{\int_0^{\tau} (m\mathbf g(B_t)-\gamma(\lambda)) \mathrm{d}t}, B_\tau=1 \right]<\infty.
$$
Thus for any $\lambda\in\R$, the gauge function
\begin{equation}\label{Gauge}
g(x,\lambda):=\Pi_x \left[ e^{\int_0^{\tau}
(m\mathbf g(B_t)-\gamma(\lambda)) \mathrm{d}t}\right]
\end{equation}
is bounded for $x\in[0,1]$.
$\gamma(\lambda)$ is analytic, convex and even, so $\gamma(0)$ is the minimum of value of $\gamma(\lambda)$ and for any $\lambda\neq 0$, $\gamma(\lambda) > \gamma(0)$. Hence,
\begin{equation}
\widehat  \psi(x,\lambda)\le g(x,\lambda)\le g(x,0)=\Pi_x \left[ e^{\int_0^{\tau} (m\mathbf g(B_t)-\gamma(0)) \mathrm{d}t}\right]
\leq\sup_{x\in[0,1]}
g(x, 0)<\infty,\quad x\in[0,1].
\end{equation}
Using the dominated convergence theorem, we obtain that
$\widehat  \psi(x,\cdot)\in C(\R)$.

Now we prove $\widehat  \psi(x,\cdot)\in C^1(\R\setminus\{0\})$. We only need to deal with the case $\lambda>0$. Fix a $\lambda_0>0$.
Since $B_{\tau} e^{-\lambda B_{\tau}} e^{\int_0^{\tau} (m\mathbf g(B_t)-\gamma(\lambda)) \mathrm{d}t} \le e^{\int_0^{\tau} (m\mathbf g(B_t)-\gamma(0)) \mathrm{d}t}$, and we have
for $\lambda\in [\lambda_0/2, 3\lambda_0/2]$,
\begin{equation}\label{psi_lambda_bound}
\tau e^{\int_0^{\tau} (m\mathbf g(B_t)-\gamma(\lambda)) \mathrm{d}t}
\leq Me^{(\gamma(\lambda)-\gamma(0))\tau} e^{\int_0^{\tau} (m\mathbf g(B_t)-\gamma(\lambda)) \mathrm{d}t}
= Me^{\int_0^{\tau} (m\mathbf g(B_t)-\gamma(0)) \mathrm{d}t},		
\end{equation}
for some large constant $M$.
Thanks to the dominated convergence theorem,
$\widehat  \psi(x,\cdot)$ is differentiable at $\lambda_0$. Since $\lambda_0>0$ is arbitrary, $\widehat  \psi(x,\cdot)$  is differentiable  on $(0, \infty)$.
Taking the partial derivative of  $\widehat \psi$ with respect to $\lambda$, we obtain
\begin{equation}\label{diff-f}
\widehat  \psi_\lambda(x,\lambda) = - \Pi_x\left[  B_{\tau} e^{-\lambda B_{\tau}} e^{\int_0^{\tau} (m\mathbf g(B_t)-\gamma(\lambda)) \mathrm{d}t} \right]
-\gamma'(\lambda)\Pi_x \left[ \tau e^{-\lambda B_{\tau}} e^{\int_0^{\tau} (m\mathbf g(B_t)-\gamma(\lambda)) \mathrm{d}t} \right].
\end{equation}
Using the dominated convergence theorem,  we obtain
$\widehat  \psi_{\lambda}(x,\cdot)$ is continuous.

Using the Markov property of Brownian motion for the second term of the right hand side of \eqref{diff-f}, we have
\begin{equation*}
\widehat  \psi_\lambda(x,\lambda) = -  \Pi_x\left[  B_{\tau} e^{-\lambda B_{\tau}} e^{\int_0^{\tau} (m\mathbf g(B_t)-\gamma(\lambda)) \mathrm{d}t} \right]
-\gamma'(\lambda)\Pi_x \left[\int^\tau_0e^{\int_0^{s} (m\mathbf g(B_t)-\gamma(\lambda)) \mathrm{d}t} \widehat  \psi(B_s,\lambda) \mathrm{d}s\right].
\end{equation*}
From the classical theory of Schr\"{o}dinger equations, we obtain that
$\widehat  \psi_\lambda(\cdot,\lambda)\in C^2(\R)$, and $\widehat  \psi_\lambda$ satisfies
\begin{equation}\label{diff-lambda-f}
\frac{1}{2} \widehat  \psi_{\lambda xx} +
(m\mathbf g-\gamma(\lambda))
\widehat  \psi_{\lambda} = \gamma'(\lambda) \widehat  \psi.
\end{equation}
One can easily check that
$\widetilde\psi=e^{\lambda x}\widehat  \psi(\lambda,x)$ satisfies \eqref{diff-lambda-tildepsi}.
\end{proof}

\begin{lemma}\label{lemma_cpsi}
Suppose $\psi(\cdot,\lambda)$ is a positive eigenfunction of
\eqref{eigen} with  $\psi(0,\cdot)\in C(\R) \cap C^1(\R\setminus\{0\})$. Then for any $x\in\R$, $\psi(x,\cdot)\in  C(\R) \cap C^1(\R\setminus\{0\})$.
Moreover,
 $\psi_{\lambda}(x,\lambda)$ satisfies
\begin{align}\label{eigen_diff}
&\frac{1}{2} \psi_{\lambda xx}(x, \lambda) - \psi_x(x, \lambda) - \lambda \psi_{\lambda x}(x, \lambda) + (\frac{1}{2}\lambda^2 + m\mathbf g(x))\psi_{\lambda}(x, \lambda) + \lambda\psi (x, \lambda)\\
&= \gamma(\lambda) \psi_{\lambda} (x, \lambda) + \gamma'(\lambda) \psi(x, \lambda).\nonumber
\end{align}
\end{lemma}
\begin{proof} Since $\psi(\cdot,\lambda)$ is 1-periodic, it suffices to verify $\frac{\partial\psi(x,\lambda)}{\partial\lambda}$ exists for $x\in [0,1]$.
Define
	\begin{equation}\label{def_phi}
	\phi(x,\lambda) := e^{-\lambda x} \psi(x,\lambda), \quad x\in\R.
	\end{equation}
	Then $\phi$ satisfies
	\begin{equation}\label{phi-equ}
	\frac{1}{2} \phi_{xx}(x,\lambda) + m\mathbf g(x) \phi(x,\lambda) = \gamma(\lambda) \phi(x,\lambda), \quad x\in\R.
	\end{equation}
Therefore, $\phi$ is the solution of the following boundary value problem
\begin{equation}
\left\{\begin{array}{rl}
&\frac{1}{2} \phi_{xx}(x,\lambda) + m\mathbf g(x)\phi(x,\lambda) = \gamma(\lambda) \phi(x,\lambda), \quad x\in (0,1)  \\
&\phi(1,\lambda) = e^{-\lambda} \psi(0,\lambda), \phi(0,\lambda)=\psi(0,\lambda).\end{array}\right.
\end{equation}
By the probabilistic representation of the solution to the above boundary value problem, we have
\begin{equation}\label{Sch_phi}
\phi(x,\lambda) =  \psi(0,\lambda) \Pi_x \left[e^{-\lambda B_{\tau}} e^{\int_0^{\tau} (m\mathbf g(B_t)-\gamma(\lambda)) \mathrm{d}t} \right], \quad x\in[0,1],
\end{equation}
where $\tau = \inf\{t > 0: B_t \notin (0,1) \}$.
Let $\widehat \psi$ be defined by \eqref{def-f}.
Then
\begin{equation}\label{phi-f}
\phi(x,\lambda) =  \psi(0,\lambda) \widehat \psi(x,\lambda), \quad x\in[0,1].
\end{equation}
In Lemma \ref{tildepsi}, we proved that for any $x\in\R$,
$\widehat \psi(x,\cdot)\in  C(\R) \cap C^1(\R\setminus\{0\})$.
By \eqref{def_phi} and the assumption that $\psi(0,\lambda)\in C(\R) \cap C^1(\R\setminus\{0\})$, for any $x\in\R$, $\phi(x,\cdot)\in  C(\R) \cap C^1(\R\setminus\{0\})$.
Since $\widehat \psi$ satisfies \eqref{diff-lambda-f}, we have
\begin{equation}\label{diff-phi}
\frac{1}{2} \phi_{\lambda xx} +
m\mathbf g \phi_{\lambda}
= \gamma(\lambda) \phi_{\lambda} + \gamma'(\lambda) \phi.
\end{equation}
\eqref{eigen_diff} follows easily from \eqref{def_phi}.
\end{proof}

\begin{lemma} There is a positive eigenfunction $\psi(\cdot,\lambda)$ of
\eqref{eigen} with $\int_0^1 \psi(x,\lambda)dx = 1$ and  $\psi(x,\cdot)\in  C(\R) \cap C^1(\R\setminus\{0\})$ for $x\in\mathbb{R}$. Moreover, $\psi_{\lambda}(x,\lambda)$  satisfies \eqref{eigen_diff}.
\end{lemma}
\begin{proof}
For any $\lambda\in\R$,
let $\widetilde{\psi}(\cdot,\lambda)$ be a positive eigenfunction of the periodic problem \eqref{eigen} with  $\widetilde\psi(0,\lambda)=1$.
By Lemma \ref{tildepsi}, for any fixed $x\in[0,1]$, $\widetilde\psi(x,\cdot)\in  C(\R) \cap C^1(\R\setminus\{0\})$. Put
\begin{equation*}
\widetilde{c}(\lambda) = \int_0^1 \widetilde{\psi}(x,\lambda) \mathrm{d}x.
\end{equation*}
Define $\psi(x,\lambda):=\widetilde{c}(\lambda)^{-1}\widetilde{\psi}(x,\lambda)$. Then $\int_0^1 \psi(x,\lambda)dx = 1$. By Lemma \ref{lemma_cpsi}, we only need to prove that $\widetilde{c}(\cdot)\in  C(\R) \cap C^1(\R\setminus\{0\})$.
We only need to deal with the case $\lambda>0$. So in the remainder of this proof, we assume $\lambda>0$.
By the proof of Lemma \ref{tildepsi}, $\sup_{x\in[0,1]}\widetilde{\psi}(x,\lambda) \leq e^\lambda\sup_{x\in[0,1]}g(x,0)<\infty$.
By the bounded convergence theorem,  $\widetilde{c}(\cdot)$ is continuous in $\R$. We now show $\widetilde{c}(\cdot)\in C^1(\R\setminus\{0\})$.
By symmetry, it suffices to show  $\widetilde{c}(\cdot)\in C^1((0,\infty))$.
For any $\lambda>0$ and $\epsilon>0$ with $\lambda-\epsilon>0$, define $D_\epsilon:=\{(x,\bar{\lambda}): 0\leq x\leq 1,\, \lambda-\epsilon\leq \bar{\lambda}\leq \lambda+\epsilon \}$.
Note that
$$
\frac{\partial\widetilde\psi}{\partial\lambda}(x,\lambda)=
xe^{\lambda x}\widehat \psi(x,\lambda)+e^{\lambda x}\frac{\partial}{\partial\lambda}\widehat \psi(x,\lambda).
$$
Combining \eqref{diff-f} and \eqref{psi_lambda_bound} we know that $\widetilde{\psi}_{\lambda}(x,\lambda)$ is bounded in $D_\epsilon$. Therefore, by the bounded convergence theorem, we have for $|h|<\epsilon$
$$\lim_{h\rightarrow 0} \frac{\widetilde{c}(\lambda+h)-\widetilde{c}(\lambda)}{h}
= \int_0^1 \lim_{h\rightarrow 0} \frac{\widetilde{\psi}(x,\lambda+h) - \widetilde{\psi}(x,\lambda)}{h} \mathrm{d}x
= \int_0^1 \widetilde{\psi}_{\lambda}(x,\lambda) \mathrm{d}x.$$
Hence, $\widetilde{c}(\cdot)$ is differentiable
and $\widetilde{c}'(\lambda)= \int_0^1 \widetilde{\psi}_{\lambda}(x,\lambda) dx$. By the bounded convergence theorem again, we get that
\begin{equation*}
\lim_{h\rightarrow 0} \widetilde{c}'(\lambda+h)
= \int_0^1 \lim_{h\rightarrow 0} \widetilde{\psi}_{\lambda}(x,\lambda+h) \mathrm{d}x = \int_0^1 \widetilde{\psi}_{\lambda}(x,\lambda) \mathrm{d}x = \widetilde{c}'(\lambda).
\end{equation*}
This shows the continuity of $\widetilde{c}'(\lambda)$.
\end{proof}

\subsection{Measure change for Brownian motion}\label{ss:mcBB}

\begin{lemma}\label{lemma_Yt}
Suppose  $\{B_t, t\geq 0;\Pi_x\}$ is a  Brownian motion starting from $x\in\R$. Define
\begin{equation}\label{mart_eta-t}
 \Xi_t(\lambda):= e^{-\gamma(\lambda)t - \lambda B_t + m\int_0^t\mathbf g(B_s)\mathrm{d}x} \psi(B_t,\lambda).
\end{equation}
Then $\{\Xi_t(\lambda), t\geq 0\}$  is a $\Pi_x$-martingale.
\end{lemma}
\begin{proof} $\psi(x,\lambda)$ is strictly positive, so by It\^{o}'s formula, we obtain
$$
    \frac{\psi(B_t,\lambda)}{\psi(B_0,\lambda)}
	= \exp\bigg{\{}  \log\psi(B_t,\lambda) - \log\psi(B_0,\lambda) \bigg{\}}
	= \exp\bigg{\{}  \int_0^t \frac{\psi_x}{\psi} \mathrm{d}B_s + \frac{1}{2} \int_0^t \frac{\psi_{xx}\psi-\psi_x^2}{\psi^2} \mathrm{d}x \bigg{\}},
$$
	where we have written $\psi = \psi(B_t,\lambda)$, $\psi_x = \psi_x(B_t,\lambda)$ and $\psi_{xx} = \psi_{xx}(B_t,\lambda)$ for short. Thus,
	\begin{align}\label{mart_YGir}
    \frac{\Xi_t(\lambda)}{\Xi_0(\lambda)}
    &=\exp\bigg{\{}  \int_0^t \left(\frac{\psi_x}{\psi}-\lambda\right) \mathrm{d}B_s +  \int_0^t \left(\frac{\psi_{xx}\psi-\psi_x^2}{2\psi^2} + m\mathbf g(B_s)-\gamma(\lambda)\right)\mathrm{d}s \bigg{\}}\\
	&= \exp\bigg{\{}  \int_0^t \left(\frac{\psi_x}{\psi}-\lambda\right) \mathrm{d}B_s - \frac{1}{2} \int_0^t \left(\frac{\psi_x}{\psi}-\lambda\right)^2 \mathrm{d}s \bigg{\}},\notag
	\end{align}
	where the second equality follows from \eqref{eigen}.
	For fixed $\lambda$, by  periodicity, $\frac{\psi_x}{\psi}-\lambda$ is bounded, and so Novikov's condition is satisfied.
    Therefore, $\{\Xi_t(\lambda), t\geq 0\}$  is a $\Pi_x$-martingale.

\end{proof}
	
Since $\frac{\Xi_t(\lambda)}{\Xi_0(\lambda)}$ is a non-negative martingale of mean $1$, we can define a probability measure $\Pi^\lambda_x$ by
\begin{equation}\label{meas_Plambda}
\frac{\mathrm{d}\Pi_x^{\lambda}}{\mathrm{d}\Pi_x}\bigg{|} _{\mathcal{F}_t^B} =
\frac{\Xi_t(\lambda)}{\Xi_0(\lambda)},
\end{equation}
where $\{\mathcal{F}_t^B: t\geq 0\}$
is the natural filtration of Brownian motion.
Define
\[	
\phi(x,\lambda) := e^{-\lambda x} \psi(x,\lambda),\quad \lambda\in\R,\, x\in\R.
\]
A direct calculation shows that $\phi$ satisfies
\begin{equation}
\frac{1}{2} \phi_{xx}(x,\lambda) + m\mathbf g(x) \phi(x,\lambda) = \gamma(\lambda) \phi(x,\lambda)
\end{equation}
and
\begin{equation*}
\frac{\phi_x(x,\lambda)}{\phi(x,\lambda)} = \frac{\psi_x(x,\lambda)}{\psi(x,\lambda)}-\lambda.
\end{equation*}
By \eqref{mart_YGir} and Girsanov's theorem,
$B_t-\int_0^t\frac{\phi_x(B_s,\lambda)}{\phi(B_s,\lambda)}\mathrm{d}s$ is a $\Pi_x^{\lambda}$-Brownian motion.
In other words, under $\Pi_x^{\lambda}$,
$\{B_t, t\geq 0\}$ satisfies
\begin{equation}\label{dY}	
\mathrm{d}B_t = \frac{\phi_x(B_t,\lambda)}{\phi(B_t,\lambda)}\mathrm{d}t + \mathrm{d}\widehat B_t, \quad B_0 = x,
\end{equation}
where $\{\widehat B_t,\, t\geq 0;\, \Pi_x^{\lambda} \}$ is a Brownian motion.
Hence,
$\{B_t, \Pi_x^{\lambda}\}$ is a diffusion with infinitesimal generator
\begin{equation}\label{Y_infin}
(\mathcal{A}f)(x) = \frac{1}{2} \frac{\partial^2 f(x)}{\partial x^2} + \left(\frac{\psi_x(x,\lambda)}{\psi(x,\lambda)}-\lambda\right)\frac{\partial f(x)}{\partial x}.	
\end{equation}
Since $\frac{\phi_x(\cdot,\lambda)}{\phi(\cdot,\lambda)}$ is 1-periodic, the law of $\{B_t-B_0:t\geq 0 \}$ under $\Pi_x^{\lambda}$ is the same as under $\Pi_{x+1}^{\lambda}$.

{\it In the remainder of this paper, we always assume
that $\{Y_t, t\geq 0; \Pi_x^{\lambda}\}$ is a diffusion with infinitesimal generator \eqref{Y_infin}.}
To prove Theorem \ref{thrm1}, we need some properties of $Y_t$ under $\Pi_x^{\lambda}$.
\cite[Lemma 2.6 and Corollary 2.7]{LTZ}
obtained the strong law of large numbers of $\{Y_t\}$ under $\Pi_x^{\lambda^*}$ for binary branching in the critical case $\lambda=\lambda^*$ and their proofs also work for any $\lambda\in\R$,
and thus $Y_t/t\rightarrow -\gamma'(\lambda)$, $\Pi_x^{\lambda}$-almost surely. Now we prove the analog for our general case.
	
\begin{lemma}\label{lemma_ldp}
Let $\mu_t$ be the law of $\{\frac{Y_t}{t} \}$ under $\Pi_x^{\lambda}$. Then $\{\mu_t \}$ satisfies a large deviation principle with good rate function
\begin{equation}
I(z) := \gamma^*(z) + \{\lambda z + \gamma(\lambda) \},
\end{equation}
where $\gamma^*(z) = \sup_{\eta\in\R} \{\eta z - \gamma(\eta) \}$ denotes the Fenchel-Legendre transform of $\gamma$.
\end{lemma}
\begin{proof}
For any $x\in\R$, $\lambda>0$, $t>0$ and any measurable function $F:C[0,t]\rightarrow\R$, we have
\begin{equation}\label{Yequ1}
\Pi_x^{\lambda}\left[ F(\{Y_s \}_{s\leq t})  \right] = \Pi_x \bigg{[} \frac{\psi(B_t,\lambda)}{\psi(x,\lambda)} e^{-\gamma(\lambda)t-\lambda (B_t-x) + m\int_0^t\mathbf g(B_s)\mathrm{d}s} F(\{B_s \}_{s\leq t})  \bigg{]}.
\end{equation}
Since $\psi$ is strictly positive and bounded, setting $F\equiv 1$ in the previous equation, we get that	\begin{equation}\label{Bequ2}
\Pi_x\left[ e^{-\lambda B_t + m\int_0^t\mathbf g(B_s)\mathrm{d}s}   \right] = e^{-\lambda x + \gamma(\lambda)t + O(1)},	
\end{equation}
where for fixed $\lambda$, $O(1)$ is uniformly bounded.
	
By \eqref{Yequ1}, for any $\eta\in\R$,
\begin{align*}
\frac{1}{t} \log \Pi_x^{\lambda}[e^{\eta Y_t}] &= \frac{1}{t} \log \Pi_x \bigg{[} \frac{\psi(B_t,\lambda)}{\psi(x,\lambda)} e^{-\gamma(\lambda)t-\lambda (B_t-x) + m\int_0^t\mathbf g(B_s)\mathrm{d}s}  e^{\eta B_t} \bigg{]}\\
&= \frac{1}{t} \log \Pi_x \bigg{[} \frac{\psi(B_t,\lambda)}{\psi(x,\lambda)} e^{-(\lambda-\eta) B_t + m\int_0^t\mathbf g(B_s)\mathrm{d}s}   \bigg{]} + \frac{\lambda x}{t}- \gamma(\lambda).
\end{align*}
It follows from \eqref{Bequ2} that as $t\to\infty$,
\begin{equation}
\lim_{t\rightarrow\infty} \frac{1}{t} \log \Pi_x^{\lambda}[e^{\eta Y_t}] = \gamma(\lambda-\eta) - \gamma(\lambda).
\end{equation}
The Fenchel-Legendre transform of $\eta\mapsto\gamma(\lambda-\eta)-\gamma(\lambda)$ is given by
\begin{align}\label{I(z)}
I(z) :&= \sup_{\eta\in\R} \big{\{} \eta z - [\gamma(\lambda-\eta)-\gamma(\lambda)]\big{\}}\\
&= \sup_{\eta\in\R} \big{\{} (\eta-\lambda) z - \gamma(\lambda-\eta)]\big{\}} + \lambda z  + \gamma(\lambda) \notag \\
&= \gamma^*(z) + \{\lambda z + \gamma(\lambda) \}, \notag
\end{align}
where the last equality follows from the fact that $\gamma$ is an even function.
Note that $\gamma(\lambda-\eta) - \gamma(\lambda)$ is
differentiable with respect to $\eta$ by Lemma \ref{lemma_gamma}.
By the G\"{a}rtner-Ellis theorem (for example, see \cite[\S 2.3]{DZ}),  $\{\mu_t \}$ satisfies a large deviation principle with good rate function $I(z)$.
\end{proof}

\begin{lemma}\label{lemma_slln}
For any $x\in\R$, it holds that $\frac{Y_t}{t} \rightarrow -\gamma'(\lambda)$ $\Pi_x^{\lambda}$-almost surely.
\end{lemma}
\begin{proof}
By Lemma \ref{lemma_gamma} (1),  $\gamma'(\lambda)$ is strictly increasing. Next, note that $I(z) =\sup_{\eta\in\R} \big{\{} \eta z - [\gamma(\lambda-\eta)-\gamma(\lambda)]\big{\}}$ and that the
derivative of $\eta z - \gamma(\lambda-\eta)$ with respect to $\eta$ is
$z+\gamma'(\lambda-\eta)$. For fixed $z$, define $\eta_z$ such that $\gamma'(\lambda-\eta_z) = -z$.
Here the existence of $\eta_z$ is guaranteed by \eqref{gamma'}. By the convexity of $\gamma$, $\eta_z$ is an increasing function of $z$. Moreover, because $\gamma'$ is strictly increasing, we have
\begin{equation}\label{I(z)anoth}
	I(z) = -\eta_z \gamma'(\lambda-\eta_z) - [\gamma(\lambda-\eta_z)-\gamma(\lambda)],
\end{equation}
where $I(z)$ is equal to $0$ when $\eta_z=0$ or
equivalently $z=-\gamma'(\lambda)$. Thus
$\frac{\mathrm{d}I(z)}{\mathrm{d}\eta_z} = \eta_z \gamma''(\lambda-\eta_z)\geq 0$
(respectively $\frac{\mathrm{d}I(z)}{\mathrm{d}\eta_z}\leq 0$) when $\eta_z>0$ (respectively $\eta_z<0$). Using \eqref{I(z)anoth} and the fact that $\gamma'$ is strictly increasing, we have for any $\epsilon>0$,
\begin{equation}
	\delta := \inf\{I(z): |z+\gamma'(\lambda)|\geq \epsilon \} = I(-\gamma'(\lambda)-\epsilon) \wedge I(-\gamma'(\lambda)+\epsilon) > 0.
\end{equation}

Applying the large deviation principle of $\{\mu_t \}$, we get
\begin{equation}
\Pi_x^{\lambda}\left( \left|\frac{Y_t}{t}+\gamma'(\lambda)\right| > \epsilon \right) \leq Ce^{-\delta t/2}.
\end{equation}
By \eqref{Yequ1}, there is a constant $C_1(\lambda)>0$ such that
\begin{align*}
	P(x,T) :&= \Pi_x^{\lambda} \left( \max_{t\in[0,1]} |Y_t-Y_0+\gamma'(\lambda)t| > T\epsilon \right) \\
	&= \Pi_x \left[\frac{\psi(B_1,\lambda)}{\psi(x,\lambda)} e^{-\gamma(\lambda)-\lambda (B_1-x) + m\int_0^1\mathbf g(B_s)\mathrm{d}s} \mathbf{1}_{\{\max_{t\in[0,1]} |B_t-x+\gamma'(\lambda)t| > T\epsilon\}}  \right]\\
	&\leq C_1(\lambda)\Pi_x \left[ e^{-\lambda (B_1-x)} \mathbf{1}_{\{\max_{t\in[0,1]} |B_t-x+\gamma'(\lambda)t| > T\epsilon\}}  \right]\\
	&\leq C_1 (\lambda)\Pi_0\left[ e^{\lambda B_1^*}
	\mathbf{1}_{\{B_1^* > T\epsilon - |\gamma'(\lambda)|\}} \right],
	\quad x\in[0,1],
\end{align*}	
where $B_1^*= \max_{t\in[0,1]}\{|B_t|\}$. Since, under $\Pi_0$, $B_1^*$
has the same distribution as $|B_1|$ and  $\int_{x}^{\infty} e^{-\frac{y^2}{2}} \mathrm{d}y \leq \frac{1}{x} e^{-\frac{x^2}{2}}$ for $x>0$, there is a constant $C_2(\lambda)>0$ such that
\begin{equation}
    P(x,T)\leq C_2(\lambda) e^{-\delta T/2},\quad x\in[0,1].
\end{equation}
Thus, for any $n\in\N$ and $x\in\R$, one has
\begin{align*}
	&\Pi_x^{\lambda} \bigg{(} \bigcup_{t\in [n,n+1]} \{|Y_t+\gamma'(\lambda)t| > 2t\epsilon \}  \bigg{)}\\
	\leq &\Pi_x^{\lambda} \left( |Y_n+\gamma'(\lambda)n| > n\epsilon \right) + \Pi_x^{\lambda} \left( \max_{t\in[0,1]} |Y_{n+t}-Y_n+\gamma'(\lambda)t| > n\epsilon \right) \\
	\leq &Ce^{-n\delta/2} + \Pi_x^{\lambda} \left[ P(Y_n,n) \right]
    \leq Ce^{-n\delta/2},
\end{align*}
where the last inequality follows from the fact that $P(x,T)$ is 1-periodic in $x$.
Since $\epsilon>0$ is arbitrary, by the Borel-Cantelli lemma, we obtain
 $\frac{Y_t}{t} \rightarrow -\gamma'(\lambda)$ $\Pi_x^{\lambda}$-almost surely.
\end{proof}

\begin{lemma}\label{mart_deri_BM}
Suppose $\{B_t, t\geq 0; \Pi_x\}$ is a Brownian motion starting from $x\in \R$.
Define
\begin{equation}\label{mart_St}
\Upsilon_t(\lambda):
= e^{-\gamma(\lambda)t-\lambda B_t + m\int_0^t\mathbf g(B_s)\mathrm{d}s} \left( \psi(B_t,\lambda)(\gamma'(\lambda)t+B_t) - \psi_{\lambda}(B_t,\lambda) \right).
\end{equation}
Then $\{\Upsilon_t(\lambda),  t\geq 0\}$ is a $\Pi_x$-martingale.
\end{lemma}
\begin{proof} For convenience,
put $J_t:= e^{-\gamma(\lambda)t-\lambda B_t + m\int_0^t \mathbf g(B_s)\mathrm{d}s}$.
A straightforward computation using It\^{o}'s formula yields
$$
\mathrm{d}J_t = -\lambda J_t \mathrm{d}B_t +\left(\frac{1}{2}\lambda^2 - \gamma(\lambda) + m\mathbf g(B_t)\right) J_t \mathrm{d}t,
$$
\begin{align*}
&\mathrm{d}\left(\psi(B_t,\lambda)(\gamma'(\lambda)t+B_t) - \psi_{\lambda}(B_t,\lambda) \right)\\
=& \left[  \psi_x(B_t,\lambda)(\gamma'(\lambda)t+B_t) + \psi(B_t,\lambda) -\psi_{\lambda x} (B_t,\lambda) \right] \mathrm{d}B_t \\
& +\left[\frac{1}{2}\psi_{xx}(B_t,\lambda)(\gamma'(\lambda)t+B_t) + \psi(B_t,\lambda)\gamma'(\lambda) + \psi_x(B_t,\lambda) - \frac{1}{2}\psi_{\lambda xx}(B_t,\lambda)  \right] \mathrm{d}t,
\end{align*}
and
\begin{align*}
&\mathrm{d}\left[J_t (\psi(B_t,\lambda)(\gamma'(\lambda)t+B_t) - \psi_{\lambda}(B_t,\lambda) \right] \\
= &J_t \left[(\psi_x - \lambda\psi)(B_t,\lambda)(\gamma'(\lambda)t+B_t) + \lambda\psi_{\lambda} (B_t,\lambda)+ \psi(B_t,\lambda) - \psi_{x\lambda}(B_t,\lambda)\right] \mathrm{d}B_t\\
& + J_t(\gamma'(\lambda)t+B_t) \left( \frac{1}{2} \psi_{xx} - \lambda \psi_x + \left(\frac{1}{2}\lambda^2-\gamma(\lambda)+m\mathbf g(B_t)\right)\psi  \right) (B_t,\lambda)\mathrm{d}t\\
& - J_t \left( \frac{1}{2} \psi_{\lambda xx} - \psi_x - \lambda \psi_{\lambda x} + \left(\frac{1}{2}\lambda^2 -\gamma(\lambda) + m\mathbf g(B_t)\right)\psi_{\lambda} + \lambda\psi - \gamma'(\lambda) \psi \right) (B_t,\lambda)\mathrm{d}t\\
= &J_t \left[ (\gamma'(\lambda)t+B_t)(\psi_x-\lambda\psi) + \lambda\psi_{\lambda} + \psi - \psi_{\lambda x} \right](B_t,\lambda) \mathrm{d}B_t,
\end{align*}
where in the last equality we used \eqref{eigen} and \eqref{eigen_diff}.
Note that $\psi$, $\psi_{\lambda}$ and $\psi_{\lambda x}$ are 1-periodic in $x$, so they are bounded for fixed $\lambda\in\R$. Therefore,
$\{\Upsilon_t(\lambda), t\geq 0\}$ is a $\Pi_x$-martingale.
\end{proof}

The martingale
$\{(\Upsilon_t(\lambda))_{t\geq 0}, \Pi_x\}$
may take negative values.
Now we introduce a related non-negative martingale.
Before giving its definition, we first give some  properties of the function $h$ defined by
\begin{equation}\label{def-h}
h(x):= x - \frac{\psi_{\lambda}(x,\lambda)}{\psi(x,\lambda)}.
\end{equation}
Clearly $h(x)$ is continuous and satisfies $h(x+1)=h(x)+1$.
Recall that we always suppose $\lambda>0$ unless explicitly stated otherwise.
Since $\phi(x, \lambda) = e^{-\lambda x} \psi(x,\lambda)$, we have
\begin{equation}\label{def-h2}
h(x)=-\frac{\phi_{\lambda}(x,\lambda)}{\phi(x,\lambda)}.
\end{equation}
It is easy to see that $h'$ is 1-periodic and continuous.
Thus, $h'$ is bounded.

\begin{lemma}\label{lemma_h_incre}
$h'$ is strictly positive.
\end{lemma}
\begin{proof}
Recall that $\phi(x, \lambda) = e^{-\lambda x} \psi(x,\lambda)$ satisfies
\begin{equation}
\frac{1}{2} \phi_{xx}(x,\lambda) + m\mathbf g(x) \phi(x,\lambda) = \gamma(\lambda) \phi(x,\lambda).
\end{equation}
For any $-\infty\leq y<z \leq \infty$, define
\begin{equation}
\tau_{(y,z)} := \inf\{t>0: B(t) \notin (y,z) \},
\end{equation}
and $\tau_y := \inf\{t>0: B(t) = y\}$.
By the probabilistic representation for the solution of the Schr\"{o}dinger equation, we have for any $x\in (y,z)$,
\begin{align*}
\phi(x,\lambda) &= \Pi_x  \left[ \phi(B_{\tau_{(y,z)}},\lambda) e^{\int_0^{\tau_{(y,z)}} (m\mathbf g(B_t)-\gamma(\lambda)) \mathrm{d}t} \right]\\
&= \Pi_x \left[  \phi(y,\lambda) \mathbf{1}_{\{\tau_y<\tau_z \}} e^{\int_0^{\tau_y} (m\mathbf g(B_t)-\gamma(\lambda)) \mathrm{d}t} \right] + \Pi_x \left[  \phi(z,\lambda) \mathbf{1}_{\{\tau_z<\tau_y \}} e^{\int_0^{\tau_z} (m\mathbf g(B_t)-\gamma(\lambda)) \mathrm{d}t} \right].
\end{align*}
By the monotone convergence theorem, the first term on the right-hand side of the last equation converges
to
$$\Pi_x \left[  \phi(y,\lambda) e^{\int_0^{\tau_y} (m\mathbf g(B_t)-\gamma(\lambda)) \mathrm{d}t} \right]
$$
as $z\rightarrow\infty.$
The second term of the last equation is equal to
$$
e^{-\lambda z} \Pi_x \left[  \psi(z,\lambda) \mathbf{1}_{\{\tau_z<\tau_y \}} e^{\int_0^{\tau_z} (m\mathbf g(B_t)-\gamma(\lambda)) \mathrm{d}t} \right],
$$
thus bounded by $Ce^{-\lambda z}$. So it converges to zero when $z\rightarrow\infty$. Therefore we have shown for any $x>y$,
\begin{equation}
\phi(x,\lambda) = \Pi_x \left[ \phi(y,\lambda) e^{\int_0^{\tau_y} (m\mathbf g(B_t)-\gamma(\lambda)) \mathrm{d}t}\right].
\end{equation}
Hence,
\begin{equation}
\ln\phi(x,\lambda) = \ln\phi(y,\lambda) + \ln \Pi_x \left[  e^{\int_0^{\tau_y} (m\mathbf g(B_t)-\gamma(\lambda)) \mathrm{d}t} \right].
\end{equation}
Differentiating both sides of the previous equation with respect to $\lambda$ gives
\begin{equation}\label{h_diff}
\frac{\phi_{\lambda}(x,\lambda)}{\phi(x,\lambda)} = \frac{\phi_{\lambda}(y,\lambda)}{\phi(y,\lambda)} - \frac{\Pi_x \left[ \gamma'(\lambda) \tau_y e^{\int_0^{\tau_y} (m\mathbf g(B_t)-\gamma(\lambda)) \mathrm{d}t} \right]}{\Pi_x \left[  e^{\int_0^{\tau_y} (m\mathbf g(B_t)-\gamma(\lambda)) \mathrm{d}t} \right]} < \frac{\phi_{\lambda}(y,\lambda)}{\phi(y,\lambda)}.
\end{equation}
In other words, $\phi_{\lambda}(\cdot,\lambda)/\phi(\cdot,\lambda)$ is strictly decreasing. Thus by \eqref{def-h2}, $h$ is strictly increasing.

By the proof of Lemma \ref{tildepsi}, we know the gauge function $g(\cdot,\lambda)$ is bounded. Therefore,
\begin{equation}
	C = \max_{y\in[x-1,x]} \Pi_x \left[  e^{\int_0^{\tau_y} (m\mathbf g(B_t)-\gamma(\lambda)) \mathrm{d}t}\right] < +\infty.
\end{equation}
It is well known, see \cite[Theorem 8.5.7]{Durrett} for example, that for any $b>0$,
\begin{equation}
\Pi_x [e^{-b\tau_y}] = e^{-(x-y)\sqrt{2b}}.
\end{equation}
Differentiating both sides of the previous equation with respect to $b$, we get
\begin{equation}
\Pi_x \left[\tau_y e^{-b\tau_y}\right] = \frac{x-y}{\sqrt{2b}} e^{-(x-y)\sqrt{2b}}.
\end{equation}
Recall that $0<\alpha \leq \min_{x\in [0,1]} \mathbf g(x)$ and \eqref{gamma_range} implies $\gamma(\lambda)>m\alpha$. Thus we have
\begin{align}
h(x)-h(y) &\geq  \Pi_x \left[ \gamma'(\lambda) \tau_y e^{\int_0^{\tau_y} (m\mathbf g(B_t)-\gamma(\lambda)) \mathrm{d}t} \right]/C\\
&\geq  \Pi_x \left[ \gamma'(\lambda) \tau_y e^{(m\alpha-\gamma(\lambda)) \tau_y} \right]/C\\
&= \frac{x-y}{C\sqrt{2(\gamma(\lambda)-m\alpha)}} \gamma'(\lambda) e^{-(x-y)\sqrt{2(\gamma(\lambda)-m\alpha)}}.
\end{align}
Therefore
\begin{align}
h'(x) = \lim_{y\rightarrow x} \frac{h(x)-h(y)}{x-y} \geq \frac{\gamma'(\lambda)} {C\sqrt{2(\gamma(\lambda)-m\alpha)}} > 0.
\end{align}
This completes the proof.
\end{proof}

For any $x\in\R$, define
\begin{equation}\label{def_tauB}
\tau^x_{\lambda}:= \inf\left\{t\geq 0: \; h(B_t) \leq - x - \gamma'(\lambda) t  \right\}.
\end{equation}
Since  $h(x)$ is strictly increasing, for any $x\in\R$,  we may rewrite the definition of $\tau_{\lambda}^x$ as
\begin{align*}
\tau^x_{\lambda} = \inf\left\{t\geq 0: B_t \leq h^{-1}(-x-\gamma'(\lambda)t) \right\}.
\end{align*}
Hence, $\tau^x_{\lambda}$ is an $\{\F_t^B \}$-stopping time.

Define
\begin{equation}\label{mart_Lambda}
\Lambda_t^{(x,\lambda)}:= e^{-\gamma(\lambda)t - \lambda B_t + m\int_0^t\mathbf g(B_s)\mathrm{d}s} \psi(B_t,\lambda)
\left( x+\gamma'(\lambda)t + h(B_t) \right)
\textbf{1}_{\{\tau^x_{\lambda} > t \}}.
\end{equation}
$\Lambda_t^{(x,\lambda)}$ is nonnegative.
We now prove that $\{\Lambda_t^{(x,\lambda)},t\geq 0\}$ is a martingale.

\begin{lemma}
For any $x, y\in \R$ with $y>h^{-1}(-x)$, $\{\Lambda_t^{(x,\lambda)},t\geq 0\}$ is a $\Pi_y$-martingale.
\end{lemma}
\begin{proof}
By Lemmas \ref{lemma_Yt} and \ref{mart_deri_BM},
$\{\Xi_t(\lambda):t\geq 0\}$ and
$\{\Upsilon_t(\lambda):t\geq 0\}$
are $\Pi_y$-martingales.
Put $F(B_t): = \Upsilon_t(\lambda) + x \Xi_t(\lambda)$.
Then $\{F(B_t):t\geq 0 \}$ is a $\Pi_y$-martingale.
The fact $\tau^x_{\lambda}$ is an $\{\F_t^B \}$-stopping time yields that $\{F(B_{\tau^x_{\lambda}\wedge t}):t\geq 0\}$ is a
$\Pi_y$-martingale.
Note that
\begin{equation*}
\Lambda_t^{(x,\lambda)} =
\left(\Upsilon_t(\lambda) + x\Xi_t(\lambda)\right)\textbf{1}_{\{\tau^x_{\lambda} > t \}}
= F(B_t)\textbf{1}_{\{\tau^x_{\lambda} > t \}}
\end{equation*}
and
\begin{equation*}
F(B_{\tau^x_{\lambda}}) = 0.
\end{equation*}
Thus
\begin{equation*}
F(B_{\tau^x_{\lambda}\wedge t}) = F(B_t)\textbf{1}_{\{\tau^x_{\lambda} > t \}} + F(B_{\tau^x_{\lambda}})\textbf{1}_{\{\tau^x_{\lambda} \leq t \}} = \Lambda_t^{(x,\lambda)},
\end{equation*}
and hence	$\{\Lambda_t^{(x,\lambda)},t\geq 0\}$ is a $\Pi_y$-martingale.
\end{proof}

For $x, y\in \R$ with $y>h^{-1}(-x)$, consider a new probability measure $\Pi^{(x, \lambda)}_y$ defined by
\begin{equation}\label{meas_change}
\frac{\mathrm{d}\Pi_y^{(x,\lambda)}}{\mathrm{d}\Pi_y}\bigg{|} _{\mathcal{F}_t^B} = \frac{\Lambda_t^{(x,\lambda)}}{\Lambda_0^{(x,\lambda)}}.
\end{equation}
Recall that under $\Pi^\lambda_x$, $\{Y_t\}$ satisfies $\mathrm{d}Y_t = \frac{\phi_x(Y_t,\lambda)}{\phi(Y_t,\lambda)}\mathrm{d}t + \mathrm{d}\widehat B_t$, where $\{\widehat B_t, t\ge 0; \Pi^\lambda_x\}$ is a  Brownian motion starting from $x$.
The following lemma, which says that
$h(Y_t)+\gamma'(\lambda)t$
is a martingale under $\Pi^\lambda_x$, is crucial to study the behavior of the motion under $\Pi_y^{(x,\lambda)}$.

\begin{lemma}\label{lemma_mart_h}
Define
\begin{equation}
M_t: = \gamma'(\lambda)t + h(Y_t) -h(Y_0), \quad t\ge 0.
\end{equation}
Then $\left\{M_t, t\geq 0; \Pi^\lambda_x\right\}$ is a martingale. Moreover, there exist two constants $c_2>c_1>0$
such that the quadratic variation $\langle M \rangle_t\in [c_1t, c_2t]$.
\end{lemma}
\begin{proof}
By \eqref{def-h2}, $h(Y_t)=-\frac{\phi_{\lambda}(Y_t,\lambda)}{\phi(Y_t,\lambda)}$.
Using \eqref{dY} and It\^{o}'s formula, we have
\begin{align*}
&\mathrm{d}\left(\frac{\phi_{\lambda}(Y_t,\lambda)}{\phi(Y_t,\lambda)}\right) = \frac{\partial}{\partial x} \left(\frac{\phi_{\lambda}(Y_t,\lambda)}{\phi(Y_t,\lambda)}\right) \mathrm{d}Y_t + \frac{1}{2} \frac{\partial^2}{\partial x^2} \left(\frac{\phi_{\lambda}(Y_t,\lambda)}{\phi(Y_t,\lambda)}\right) \mathrm{d}t\\
&= \frac{\partial}{\partial x} \left(\frac{\phi_{\lambda}(Y_t,\lambda)}{\phi(Y_t,\lambda)}\right) \mathrm{d}\widehat B_t + \bigg{[}\frac{\phi_x(Y_t,\lambda)}{\phi(Y_t,\lambda)} \frac{\partial}{\partial x} \left(\frac{\phi_{\lambda}(Y_t,\lambda)}{\phi(Y_t,\lambda)}\right) + \frac{1}{2} \frac{\partial^2}{\partial x^2} \left(\frac{\phi_{\lambda}(Y_t,\lambda)}{\phi(Y_t,\lambda)}\right) \bigg{]} \mathrm{d}t,
\end{align*}
where $\{\widehat{B}_t, t\geq 0; \Pi_x^{\lambda} \}$ is a Brownian motion.
Notice that
\begin{align}\label{phi_lambda_eq}
\frac{1}{2} \phi_{xx}(x,\lambda) + m\mathbf g(x)\phi(x,\lambda) &= \gamma(\lambda)\phi(x,\lambda), \notag \\
\frac{1}{2} \phi_{\lambda xx}(x,\lambda) + m\mathbf g(x) \phi_{\lambda}(x,\lambda) &= \gamma'(\lambda)\phi(x,\lambda) + \gamma(\lambda)\phi_{\lambda}(x,\lambda).
\end{align}
Thus
\begin{align*}
\frac{\phi_x}{\phi} \frac{\partial}{\partial x} \frac{\phi_{\lambda}}{\phi} + \frac{1}{2} \frac{\partial^2}{\partial x^2} \frac{\phi_{\lambda}}{\phi} &= \frac{\phi_x}{\phi} \frac{\phi_{\lambda x}\phi - \phi_{\lambda}\phi_x}{\phi^2} + \frac{1}{2} \frac{\phi_{\lambda xx}}{\phi} - \frac{\phi_{\lambda x}\phi_x}{\phi^2} +  \frac{\phi_{\lambda}\phi_x^2}{\phi^3} - \frac{1}{2} \frac{\phi_{\lambda}\phi_{xx}}{\phi^2}\\
&= \frac{1}{2} \frac{\phi_{\lambda xx}}{\phi} - \frac{1}{2} \frac{\phi_{\lambda}\phi_{xx}}{\phi^2}\\
&= \frac{\gamma'(\lambda)\phi + \gamma(\lambda)\phi_{\lambda} - m\mathbf g \phi_{\lambda}}{\phi} - \frac{1}{2} \frac{\phi_{\lambda}\phi_{xx}}{\phi^2}\\
&= \gamma'(\lambda) - \frac{\phi_{\lambda}}{\phi^2} \left[ \frac{1}{2}\phi_{xx} + (m\mathbf g-\gamma(\lambda)) \phi \right]
 = \gamma'(\lambda).
\end{align*}
This yields
\begin{equation}\label{mart_h}
\mathrm{d}\left(h(Y_t)\right) =  h'(Y_t) \mathrm{d}\widehat B_t -\gamma'(\lambda) \mathrm{d}t.
\end{equation}
It follows from the boundedness of $h'$ and Lemma \ref{lemma_h_incre} that $h'(x)\in [\sqrt{c_1},\sqrt{c_2}]$ for two constants $c_1\leq c_2$.
Integrating both sides of \eqref{mart_h} gives
\begin{equation*}
h(Y_t)- h(Y_0) = \int_0^th'(Y_s) \mathrm{d}\widehat B_s- \gamma'(\lambda) t.
\end{equation*}
Hence,
\begin{equation*}
M_t =\int_0^t h'(Y_s) \mathrm{d}\widehat B_s
\end{equation*}
is a martingale with quadratic variation
\begin{equation}\label{quadratic}
\langle M \rangle_t = \int_0^t \left( h'(Y_s) \right)^2 \mathrm{d}s \in [c_1t, c_2t].
\end{equation}
\end{proof}

Let $T(s) = \inf\{t>0:\langle M \rangle_t>s \}$.
Thanks to the Dambis-Dubins-Schwarz theorem, $\widetilde{B}_t := M_{T(t)}$ is a standard Brownian motion.
Note that for $y>h^{-1}(-x)$,
\begin{align}
\frac{\mathrm{d}\Pi_y^{(x,\lambda)}}{\mathrm{d}\Pi^{\lambda}_y}\bigg{|} _{\mathcal{F}_t^B}= &\frac{\mathrm{d}\Pi_y^{(x,\lambda)}}{\mathrm{d}\Pi^{}_y}\bigg{|} _{\mathcal{F}_t^B}\times \frac{\mathrm{d}\Pi_y}{\mathrm{d}\Pi^{\lambda}_y}\bigg{|} _{\mathcal{F}_t^B} \\
= &\frac{x+\gamma'(\lambda)t + h(Y_t) }{x+h(y)}\textbf{1}_{\{\forall s\leq t: \; x +h(y) + M_s > 0 \}}  \\
= &\frac{x+h(y) + M_t}{x+h(y)} \mathbf{1}_{\{\forall s\leq t: \; x +h(y) + M_s > 0 \}}.
\end{align}
Put
\begin{align}
\widetilde{\Lambda}_t &= \Lambda_{T(t)}^{(x,\lambda)}  = \left(x+h(y) + M_{T(t)}\right) \mathbf{1}_{\{\forall s\leq t: \; x +h(y) + M_{T(s)} > 0 \}}\\
&= \left(x +h(y) + \tilde{B}_t\right) \mathbf{1}_{\{\forall s\leq t: \; x+ h(y) + \tilde{B}_s > 0 \}}
\end{align}
and $\mathcal{G}_t = \mathcal{F}_{T(t)}^B$. We have
\begin{align}
\frac{\mathrm{d}\Pi_y^{(x,\lambda)}}{\mathrm{d}\Pi^{\lambda}_y}
\bigg{|} _{\mathcal{G}_t}
= \frac{x+h(y) + M_{T(t)}}{x+h(y)} \mathbf{1}_{\{\forall s\leq t: \; x+h(y) + M_{T(s)} > 0 \}},
\end{align}
that is,
\begin{align}
\frac{\mathrm{d}\Pi_y^{(x,\lambda)}}{\mathrm{d}\Pi^{\lambda}_y}\bigg{|} _{\mathcal{G}_{t}}
= \frac{\widetilde{\Lambda}_t}{\widetilde{\Lambda}_0}.
\end{align}
By \cite{Imhof},
$\left\{x +h(y)+ \tilde{B}_t, t\geq 0; \Pi_y^{(x,\lambda)}\right\}$
is a standard Bessel-3 process starting at
$x +h(y)$, i.e.,
$\left\{x +h(y)+ M_{T(t)}, t\geq 0; \Pi_y^{(x,\lambda)}\right\}$
is a standard Bessel-3 process
starting at $x +h(y)$.

\subsection{Martingales for branching Brownian motion}
In this subsection we give three martingales, that will play important roles for BBMPE.
First we prove that for any $\lambda\in\R$ and $x\in\R$, $\{W_t(\lambda),t\geq 0\}$ is a martingale.
For this, we need the following many-to-one lemma (see \cite{HR17}, \cite{LTZ} and \cite[\S 2.3]{Maillard}).

\begin{lemma}[Many-to-one Lemma]\label{l:many-to-one}
Let $t>0$ and
	$F:C[0,t]\rightarrow\R$ be a measurable function such that $F\left(B_s,s\in[0,t]\right)$ is $\F_{t}^B$-measurable. Then
	\begin{equation}\label{many-to-one}
	\E_x \left[\sum_{u\in N_t} F\left(X_u(s),s\in[0,t]\right) \right]  = \Pi_x\left[ e^{m\int_0^{t} \mathbf g(B_s) \mathrm{d}s} F\left(B_s,s\in[0,t]\right) \right] .
	\end{equation}
\end{lemma}

\begin{lemma}
For any $\lambda\in\R$ and $x\in\R$, $\left\{W_t(\lambda),t\geq 0; \P_x\right\}$ is a non-negative martingale and the limit $W(\lambda,x):= \lim_{t\uparrow\infty} W_t(\lambda)$ exists $\P_x$-almost surely.
\end{lemma}
\begin{proof}
By Lemma \ref{l:many-to-one}
and the fact that
$\{\Xi_t(\lambda): t\geq 0\}$ is a martingale, we get
\begin{equation}
	\E_xW_t(\lambda) = \Pi_x \left[e^{m\int_0^t \mathbf g(B_s) \mathrm{d}s - \gamma(\lambda)t-\lambda B_t} \psi(B_t,\lambda)\right]
    = \Pi_x\Xi_t(\lambda) = \Pi_x\Xi_0(\lambda) = e^{-\lambda x}\psi(x,\lambda).
	\end{equation}
Thus by the branching property, for $s<t$ we have
	\begin{align*}
	\E_x\left[W_t(\lambda) \;|\; \F_s\right] &= \E_x \left[ e^{-\gamma(\lambda)t} \sum_{u\in N_t} e^{-\lambda X_u(t)} \psi(X_u(t),\lambda) \;\Big|\; \F_s \right]\\
	&= e^{-\gamma(\lambda)s} \sum_{v\in N_s} \E_x \left[ e^{-\gamma(\lambda)(t-s)} \sum_{u\in N_t, u>v} e^{-\lambda X_u(t)} \psi(X_u(t),\lambda) \;\Big|\; \F_s \right]\\
	&= e^{-\gamma(\lambda)s} \sum_{v\in N_s} \E_{X_v(s)} W_{t-s}^{(v)}(\lambda)\\
	&= e^{-\gamma(\lambda)s} \sum_{v\in N_s} e^{-\lambda X_v(s)} \psi(X_v(s),\lambda)
    = W_s(\lambda),
	\end{align*}
	where $u>v$ denotes that $u$ is a descendant of $v$ and $W_{t-s}^{(v)}(\lambda)$ is the additive martingale for the BBMPE starting from $X_v(s)$.
	Therefore, $\{W_t(\lambda):t\geq 0 \}$ is a non-negative martingale and the limit $W(\lambda,x) = \lim_{t\uparrow\infty} W_t(\lambda)$ exists $\P_x$-almost surely.
\end{proof}

The second martingale is given by the following lemma.
\begin{lemma}
For any $\lambda\in\R$ and $x\in\R$, $\{\partial W_t(\lambda),t\geq 0; \P_x\}$ is a martingale.
\end{lemma}
\begin{proof}	
It follows from Lemma \ref{l:many-to-one} that
\begin{align}
\E_x \partial W_t(\lambda) =
 \Pi_x \Upsilon_t(\lambda)=\Pi_x \Upsilon_0(\lambda)
= e^{-\lambda x}\left( x\psi(x,\lambda) - \psi_{\lambda}(x,\lambda) \right).
\end{align}
Using the branching property and the Markov property, it is easy to show that for any $t>s>0$,
\begin{align*}
&\E_x\left[ \partial W_t(\lambda) \;|\; \F_s \right] \\
= &e^{-\gamma(\lambda)s} \sum_{v\in N_s} \E_{X_v(s)} \left(\partial W_{t-s}^{(v)} (\lambda) + \gamma'(\lambda)s W_{t-s}^{(v)}(\lambda) \right)\\
= &e^{-\gamma(\lambda)s} \sum_{v\in N_s} \left( e^{-\lambda X_v(s)} \left( X_v(s)\psi(X_v(s),\lambda) - \psi_{\lambda}(X_v(s),\lambda) \right) + \gamma'(\lambda)s e^{-\lambda X_v(s)} \psi(X_v(s),\lambda) \right)\\
= &\partial W_s(\lambda),
\end{align*}
where, for each $v\in N_s$, $W_{t-s}^{(v)}(\lambda)$ and $\partial W_{t-s}^{(v)}(\lambda)$ are respectively the additive and derivative martingales for the BBMPE starting from $X_v(s)$.
Therefore,  $\{\partial W_t(\lambda),t\geq 0; \P_x\}$ is a martingale.
\end{proof}

The martingale $\{\partial W_t(\lambda)\}$ may take negative values.
To study its limit we need to consider a related non-negative martingale. Note that
\begin{align}
\partial W_t(\lambda) + x W_t(\lambda) &= e^{-\gamma(\lambda)t} \sum_{u\in N_t} e^{-\lambda X_u(t)}  \psi(X_u(t),\lambda) \left( x+\gamma'(\lambda)t+
X_u(t) - \frac{\psi_{\lambda}(X_u(t),\lambda)} {\psi(X_u(t),\lambda)} \right) \\
&= e^{-\gamma(\lambda)t} \sum_{u\in N_t} e^{-\lambda X_u(t)} \psi(X_u(t),\lambda) \left(x+\gamma'(\lambda)t+
h(X_u(t))\right).
\end{align}
Put
$$
\widetilde{N}_t^x = \left\{u\in N_t: \forall s\leq t, \; x + \gamma'(\lambda)s +
h(X_u(s)) > 0
 \right\}.
$$
In the spirit of \cite{Ky}, we define for each $x\in \R$,
\begin{equation}\label{mart_V}
V_t^x(\lambda) = \sum_{u\in \widetilde{N}_t^x} e^{-\gamma(\lambda)t-\lambda X_u(t)} \psi(X_u(t),\lambda) \left( x+\gamma'(\lambda)t + h(X_u(t))\right).
\end{equation}
Now we show that $V_t^x(\lambda)$ is a martingale.

\begin{lemma}\label{lemma_Vt}
For any $x, y\in \R$ with $y>h^{-1}(-x)$, $\{V_t^x(\lambda), t\geq 0; \P_y\}$ is a martingale with respect to $\{\F_t: t\geq 0\}$.
\end{lemma}
\begin{proof} On the space-time half plane $\{(z,t):z\in\R, t\in\R^+ \}$,
consider the barrier
$\mathbf{\Gamma}^{(-x,\lambda)}$
described by $z=h^{-1}(-x-\gamma'(\lambda)t)$. By stopping the lines of descent the first time they hit this barrier we produce a random collection $\mathbf{C}(-x,\lambda)$ of individuals.
Let $\sigma_u$ denote the first hitting time for each
$u\in \mathbf{C}(-x,\lambda)$.
Consider the stopping ``line"
\begin{equation*}
\mathcal{L}(t) =\left\{u\in \mathbf{C}(-x,\lambda): \sigma_u\leq t \right\} \cup \widetilde{N}_t^x.
\end{equation*}
Let $\F_{\mathcal{L}(t)}$ be the natural filtration
generated by the spatial paths and the number of offspring of the individuals before hitting the stopping line $\mathcal{L}(t)$.
By the strong Markov branching property of  $\{Z_t, t\geq 0\}$ (see Jagers \cite[Theorem 4.14]{Jagers}, also see Dynkin \cite[Theorem 1.5]{Dynkin} for the corresponding property for superprocesses, where this property is called the special Markov property),
\begin{equation}
\E_y \left[\partial W_t(\lambda) + xW_t(\lambda) \;|\; \mathcal{F}_{\mathcal{L}(t)}  \right] = V_t^x(\lambda).
\end{equation}
Thus
\begin{equation*}
\E_y  V_t^x(\lambda) = \E_y \left[\partial W_t(\lambda) + xW_t(\lambda) \right] = e^{-\lambda y} \psi(y,\lambda) \left( x+
h(y) \right).
\end{equation*}
Hence we have for $0\leq s\leq t$,
\begin{align*}
&\E_y \left[ V_t^x(\lambda) \;|\; \mathcal{F}_s  \right]\\
= &\sum_{v\in\widetilde{N}_s^x} e^{-\gamma(\lambda)s+\lambda(\gamma'(\lambda)s+\delta)} \E_y \bigg{[} \sum_{u\in \widetilde{N}_t^x, u>v} e^{-\gamma(\lambda)(t-s)} e^{-\lambda (X_u(t)  +\gamma'(\lambda)s+\delta)} \psi(X_u(t),\lambda)  \\
&\qquad\times\left( (x-\delta) + \gamma'(\lambda)(t-s) + (\gamma'(\lambda)s + \delta)+ h(X_u(t))  \right)   \;\Big|\; \mathcal{F}_s \bigg{]}\\
= &\sum_{v\in\widetilde{N}_s^x} e^{-\gamma(\lambda)s+\lambda(\gamma'(\lambda)s+\delta)} \E_{X_v(s)+\gamma'(\lambda)s+\delta} V_{t-s}^{x-\delta}(\lambda, v)\\
= &\sum_{v\in\widetilde{N}_s^x} e^{-\gamma(\lambda)s+\lambda(\gamma'(\lambda)s+\delta)} e^{-\lambda(X_v(s)+\gamma'(\lambda)s+\delta)} \psi(X_v(s),\lambda)\times\left(x-\delta+(\gamma'(\lambda)s+\delta)+h(X_v(s))\right)\\
= &\sum_{v\in\widetilde{N}_s^x} e^{-\gamma(\lambda)s-\lambda X_v(s)} \psi(X_v(s),\lambda) \left(x+\gamma'(\lambda)s+h(X_v(s))\right)\\
= &V_s^x(\lambda),
\end{align*}
where $\delta$ is such that $(\delta+\gamma'(\lambda)s) \in \mathbb{Z}$,
and in second equality, for each $v$,
$V_{t-s}^{x-\delta}(\lambda, v)$ is the counterpart of
$V^{x-\delta}_{t-s}$
for the BBMPE starting from $X_v(s)+\gamma'(\lambda)s+\delta$ and we used the periodicity of $h$ and $\psi$.
So $\{V_t^x(\lambda), t\geq 0; \P_y\}$
is a martingale.
\end{proof}

\section{Proof of Theorem \ref{thrm1}}\label{Sec3}

\subsection{Measure change by the additive martingale}\label{spines}
	The spine decomposition theorem has been studied in many papers (for example, see \cite{BK04}, \cite{CR88},  \cite{HR17}, \cite{LPP95}, and \cite{RS}). For any $u\in N_t$, define
	\begin{equation}
	\Xi_t(u,\lambda) := e^{-\gamma(\lambda)t-\lambda X_u(t) + m\int_0^t\mathbf g(X_u(s))\mathrm{d}s} \psi(X_u(t),\lambda).
	\end{equation}
Then we may rewrite  $W_t(\lambda)$ as
	\begin{equation}\label{mart_W1}
	W_t(\lambda) = \sum_{u\in N_t} \Xi_t(u, \lambda) e^{-m\int_0^t \mathbf g(X_u(s))\mathrm{d}s}.
	\end{equation}
Define a new probability measure $\P_x^{\lambda}$ by
	\begin{equation}\label{meas_Q}
	\frac{\mathrm{d}\P_x^{\lambda}}{\mathrm{d}\P_x}\bigg{|} _{\mathcal{F}_t} = \frac{W_t(\lambda)}{W_0(\lambda)}.
	\end{equation}

Now we construct the space of Galton-Watson trees with a spine.
Here we use the same notation as those in \cite{Ky}.
Let $(\mathcal{T}, \mathcal{F}, \mathcal{F}_{t}, \P_x)$ be the filtered probability space in which BBMPE $\{Z_{t}: t \geqslant 0\}$ is defined. Let $\mathbb{T}$ be the space of Galton-Watson trees. A Galton-Watson tree $\tau \in \mathbb{T}$ is a point in the space of possible Ulam-Harris labels
	$$
	\Omega=\emptyset \cup \bigcup_{n \in \mathbb{N}}(\mathbb{N})^{n},
	$$
	where $\mathbb{N}=\{1,2,3, \ldots\}$ such that

    \begin{enumerate}[(i)]
    \item $\emptyset \in \tau$ (the ancestor);
    \item if $u, v \in \Omega$, $uv \in \tau$ implies $u \in \tau$;
    \item for all $u \in \tau,$ there exists $A_u \in\{0,1,2, \ldots\}$
    such that for $j \in \mathbb{N}$,  $j \in \tau$ if and only if $1 \leq j \leq 1+A_u$.
    \end{enumerate}
	(Here $1+A_u$ is the number of offspring of $u$,
	and $A_u$ has the same distribution as $L$.)
		
	Each particle $u\in\tau$ has a mark $(\eta_u, B_u)\in \R^+\times C(\R^+,\R)$, where $\eta_u$ is the lifetime of $u$ and $B_u$ is the motion of $u$ relative to its birth position. Then the birth time of $u$ can be written as $b_u = \sum_{v<u} \eta_v$, the death time of $u$ is $d_u = \sum_{v\leq u}\eta_v$ and the position of $u$ at time $t$ is given by $X_u(t) = \sum_{v<u} B_v(\eta_v)+B_u(t-b_u)$. We write $(\tau, B,\eta)$ as a short hand for the marked Galton-Watson tree $\{(u,\eta_u,B_u):u\in\tau \}$, and $\mathcal{T} = \{(\tau,B,\eta):\tau\in\mathbb{T} \}$. The $\sigma$-field $\mathcal{F}_{t}$ is generated by
	\begin{align*}
	\left\{\begin{array}{l}
	(u, A_u, \eta_{u},\{B_{u}(s): s \in[0, \eta_u]\}: u \in \tau \text { with } d_{u} \leq t) \text { and } \\
	(u,\{B_{u}(s): s \in[0, t-b_{u}]\}: u \in \tau \text { with } t \in[b_{u}, d_{u})): \tau \in \mathbb{T}
	\end{array}\right\}.
	\end{align*}
	A spine is a distinguished genealogical line of descent from the ancestor. A spine will be written as $\xi=\left\{\xi_{0}=\emptyset, \xi_{1}, \xi_{2}, \ldots\right\}$, where $\xi_{n} \in \tau$ is the label of $\xi$'s node in the $n$th generation. We write $u \in \xi$ if $u=\xi_{i}$ for some $i\geq0$. Now let
	\begin{equation*}
	\widetilde{\mathcal{T}}=\{(\tau, B, \eta, \xi): \xi \subseteq \tau \in \mathbb{T}\}
	\end{equation*}
	be the space of marked trees in $\mathcal{T}$ with distinguished spine, $\xi,$ let $\widetilde{\mathcal{F}}=\sigma(\widetilde{\mathcal{T}})$ and
	\begin{equation*}
	\widetilde{\mathcal{F}}_{t}=\sigma\left(\mathcal{F}_{t},\left\{(\xi: u \in \xi): u \in N_{t}\right\}\right).
	\end{equation*}
	We denote by $\{X_{\xi}(t):t\geq 0\}$ the spatial path followed by the spine $\xi$ and
	write $n=\{n_t:t\geq0 \}$ for the counting process of points of fission along the spine.

	Hardy and Harris \cite{HH09} noticed that it is convenient to consider $\{\mathbb{P}_x\}$ as measures
	on the enlarged space $(\widetilde{\mathcal{T}}, \F)$ rather than on $(\mathcal{T},\F)$.
	We extend the probability measures $\{\mathbb{P}_x\}$ to probability measures $\{\widetilde{\P}_x\}$ on $(\widetilde{\mathcal{T}}, \widetilde{\F})$. Under $\widetilde{\P}_x$, if $v$ is the particle in the $n$th generation on the spine, then for the next generation, the spine is chosen uniformly from the $1+A_v$ offspring of $v$. Therefore, we have
	\begin{equation}
	\widetilde{\P}_x(u\in\xi) = \prod_{v<u} \frac{1}{1+A_v}.
	\end{equation}
	Define
	\begin{equation}
	\tilde{\zeta}_t = \sum_{u\in N_t} \prod_{v<u} (1+A_v) \frac{\Xi_t(u,\lambda)}{\Xi_0(u,\lambda)} e^{-m\int_0^t \mathbf g(X_u(s))  \mathrm{d}s} \mathbf{1}_{\{\xi_t=u\}}.
	\end{equation}
    According to \cite{HH09} or \cite{RS}, $\{\tilde{\zeta}_t,\widetilde{\F}_t \}$ is a martingale and
	\begin{equation}
	\frac{W_t(\lambda)}{W_0(\lambda)} = \widetilde{\P}_x(\tilde{\zeta}_t \;|\; \F_t),
	\end{equation}
	in other words, $W_t(\lambda)$ is the projection of $\tilde{\zeta}_t$ onto $\F_t$.
	
	Now we define a probability measure $\widetilde{\P}_x^{\lambda}$ on $(\widetilde{\mathcal{T}}, \widetilde{\F})$ by
	\begin{equation}\label{meas_Qtilde}
	\frac{\mathrm{d}\widetilde{\P}_x^{\lambda}}{\mathrm{d}\widetilde{\P}_x}\bigg{|}_{\widetilde{\F}_t} = \tilde{\zeta}_t.
	\end{equation}
	According to \cite{HH09} or \cite{RS}, under $\widetilde{\P}_x^{\lambda}$:
	\begin{enumerate}[(i)]
		\item the ancestor starts from $x$  and the spine process $\xi$ moves according to $\Pi_x^{\lambda}$, that is, the spine moves as a diffusion with infinitesimal generator given by \eqref{Y_infin};
		\item given the trajectory $X_{\xi}$ of the spine, the branching rate is given by $(m+1)\mathbf g(X_{\xi}(t))$;
		\item at the fission time of node $v$ on the spine, the single spine particle is replaced by $1+A_v$ offspring, with $A_v$ being independent identically distributed with common distribution $\{\tilde{p}_k:k\geq 0 \}$, where $\tilde{p}_k = \frac{(k+1)p_k}{m+1}$;
		\item the spine is chosen uniformly from the $1+A_v$ offspring at the fission time of $v$;
		\item the remaining $A_v$ particles $O_v$ give rise to the independent subtrees $\{(\tau,B,\eta)_j^v\}$, $j\in O_v$, evolving as
		independent processes determined by the measure $\P_{X_v(d_v)}$ shifted to their point
		and time of creation.		
	\end{enumerate}
	Moreover, the measure $\P_x^{\lambda}$ defined by \eqref{meas_Q} satisfies
	\begin{equation}\label{P and tilde P}
	\P_x^{\lambda} = \widetilde{\P}_x^{\lambda} |_{\F}.
	\end{equation}

\subsection{Proof of Theorem \ref{thrm1}}	
	
\begin{proof}[Proof of Theorem \ref{thrm1}]
Let $\overline{W}(\lambda,x) = \limsup_{t\uparrow\infty} W_t(\lambda)$ so that $\overline{W}(\lambda,x) = W(\lambda,x) \; \P_x$-a.s.
 By \eqref{P and tilde P} and \cite[Theorem 5.3.3]{Durrett},
	\begin{align}
	\overline{W}(\lambda,x) = \infty, \;
\widetilde{\P}_x^{\lambda}\text{-a.s.}
\quad &\Longleftrightarrow \quad \overline{W}(\lambda,x) = 0, \; \P_x\text{-a.s.}\\
	\overline{W}(\lambda,x) < \infty, \;
\widetilde{\P}_x^{\lambda}\text{-a.s.}
\quad &\Longleftrightarrow \quad \int \overline{W}(\lambda,x)\mathbb{d}\P_x =1.
	\end{align}
	
	(i) When $\lambda>\lambda^*$, we have
	\begin{equation}
	W_t(\lambda) \geq C \exp\left\{ -\lambda X_{\xi}(t) - \gamma(\lambda)t \right\} = C \exp\left\{ -\lambda t \left(\frac{X_{\xi}(t)}{t} + \frac{\gamma(\lambda)}{\lambda}\right) \right\}.
	\end{equation}
	By Lemmas \ref{lemma_slln} and \ref{lemma_com}, we have
	\begin{equation}
	\lim_{t\uparrow\infty} \frac{X_{\xi}(t)}{t} + \frac{\gamma(\lambda)}{\lambda} = -\gamma'(\lambda) + \frac{\gamma(\lambda)}{\lambda} < 0.
	\end{equation}
	Thus $\overline{W}(\lambda,x) = \infty$, $\widetilde{\P}_x^{\lambda}$-a.s. and hence $W(\lambda,x) = 0$, $\P_x$-a.s.
	
	(ii)
According to the paragraph before Lemma \ref{lemma_ldp},
under $\widetilde{\P}_x^{\lambda^*}$, we have $\frac{X_{\xi}(t)}{t} \rightarrow -\gamma'(\lambda^*)$ as $t\rightarrow\infty$.
Define hitting times:
\begin{equation}
T_k := \inf\left\{t\geq 0: X_{\xi}(t) \leq x-k \right\},\quad k\in\N.
\end{equation}
Then $T_k$ is a $\widetilde{\P}_x^{\lambda^*}$-almost surely finite stopping time. Thanks to the strong Markov property and 1-periodicity, $\{T_k-T_{k-1}\}_{k\geq 1}$ are independent and identically distributed. Moreover, $T_k\rightarrow\infty$ as $k\rightarrow\infty$, $\widetilde{\P}_x^{\lambda^*}$-almost surely, and
\begin{equation}
\widetilde{\P}_x^{\lambda^*} T_1= \lim_{k\rightarrow\infty} \frac{T_k}{k} = \lim_{k\rightarrow\infty} \frac{T_k}{x-X_{\xi}(T_k)} = \frac{1}{\gamma'(\lambda^*)}.
\end{equation}
It follows that $T_k - \frac{k}{\gamma'(\lambda^*)}$ is a mean zero (non-trivial) random walk. Thus
\begin{align*}
\liminf_{t\rightarrow\infty} \left(X_{\xi}(t) + \frac{\gamma(\lambda^*)}{\lambda^*}t\right) \leq &\liminf_{k\rightarrow\infty} \left(X_{\xi}(T_k) + \gamma'(\lambda^*)T_k\right)\\
=& \liminf_{k\rightarrow\infty} \left(\frac{x-k}{\gamma'(\lambda^*)} + T_k\right) \gamma'(\lambda^*)
= -\infty,\quad \widetilde{\P}_x^{\lambda^*}\mbox{-a.s.}
\end{align*}
Thus,
$\overline{W}(\lambda^*,x) = \infty$, $\widetilde{\P}_x^{\lambda^*}$-a.s. and consequently $W(\lambda^*,x) = 0$, $\P_x$-a.s.
	
    (iii) The proof of this part is similar to that of
    \cite[Theorem 1 (iii)]{Ky}.
    Suppose that $\lambda\in [0,\lambda^*)$. Then we have
    \begin{equation}\label{limit-spine}
    \lim_{t\rightarrow\infty} \frac{X_{\xi}(t)}{t} + \frac{\gamma(\lambda)}{\lambda} = -\gamma'(\lambda) + \frac{\gamma(\lambda)}{\lambda} > 0,
    \end{equation}
    in other words, $\{X_{\xi}(t) + \frac{\gamma(\lambda)}{\lambda}t \}$ is a diffusion with strictly positive drift.
   (When $\lambda = 0$, $\{\lambda X_{\xi}(t) +\gamma(\lambda)t\} = \{\gamma(0)t\} $
   is a deterministic drift to the right and can be regarded as a degenerate diffusion. In the case, the proof below still works.)

    Suppose $\mathbf E(L\log^+L)=\infty$. Let $\{d_{\xi_i}:i\geq 0 \}$ be the fission times along the spine. Note that
	\begin{equation}\label{domi-below}
	W_{d_{\xi_k}}(\lambda) \geq C A_{\xi_k} \exp \left\{ -\lambda\left(X_{\xi}(d_{\xi_k}) + \frac{\gamma(\lambda)}{\lambda}d_{\xi_k} \right) \right\},
	\end{equation}
	where $\{A_{\xi_k}: k\geq0\}$ are iid with distribution $\{\tilde{p}_k, k\geq 0\}$. The  assumption $\mathbf E(L\log^+L)=\infty$ implies that $\widetilde{\P}_x^{\lambda} \log^+A_{\xi_k} =\infty$, and thus
    $\underset{k \rightarrow \infty}{\limsup } \; k^{-1} \log A_{\xi_k}=\infty$,
    $\widetilde{\P}_x^{\lambda}$-a.s.
	By \eqref{limit-spine} and \eqref{domi-below}, $\overline{W}(\lambda,x) = \infty$, $\widetilde{\P}_x^{\lambda}$-a.s. and hence $W(\lambda,x) = 0$, $\P_x$-a.s.
	
    Suppose $\mathbf E(L\log^+L)<\infty$. Let  $\widetilde{\mathcal{G}}$
    to be the $\sigma$-field generated by the motion of the spine and the genealogy along the spine. By the spine decomposition and the martingale property of $W_t(\lambda)$, we have
	\begin{equation}
	\widetilde{\P}_x^{\lambda} \left( W_t(\lambda)\; |\; \widetilde{\mathcal{G}} \right) = \sum_{i=1}^{n_t} A_{\xi_{i-1}} e^{-\lambda X_{\xi}(d_{\xi_{i-1}}) - \gamma(\lambda) d_{\xi_{i-1}}} \psi(X_{\xi}(d_{\xi_{i-1}}),\lambda) + e^{-\lambda X_{\xi}(t) - \gamma(\lambda)t} \psi(X_{\xi}(t),\lambda).
	\end{equation}
    The  assumption $\mathbf E(L\log^+L)<\infty$ implies that $\widetilde{\P}_x^{\lambda} \log^+A_{\xi_k} <\infty$, and thus
    $$\underset{k \rightarrow \infty}{\limsup } \; k^{-1} \log A_{\xi_k}=0.$$
    Since $\psi$  is bounded, by \eqref{limit-spine}, we have
	\begin{equation*}
	\limsup_{t\uparrow\infty} \widetilde{\P}_x^{\lambda} \left( W_t(\lambda)\; |\; \widetilde{\mathcal{G}} \right) < \infty \quad \widetilde{\P}_x^{\lambda}\text{-a.s.}
	\end{equation*}
      Hence $\liminf_{t\uparrow\infty} W_t(\lambda) < \infty$ $\widetilde{\P}_x^{\lambda}$-a.s.
	By \cite{HR09} and \eqref{P and tilde P}, $W_t(\lambda)^{-1}$ is a non-negative $\widetilde\P_x^{\lambda}$-supermartingale, which implies that the limit of  $W_t(\lambda)^{-1}$ exists  as $t\to\infty$ $\widetilde\P_x^{\lambda}$-a.s. Hence
	$\lim_{t\uparrow\infty} W_t(\lambda) < \infty$ $\widetilde{\P}_x^{\lambda}$-a.s.
	Therefore for $\lambda\in [0,\lambda^*)$ and $\mathbf E(L\log^+L)<\infty$, $W(\lambda,x)$ is a $L^1(\P_x)$-limit.	
\end{proof}

\section{Proof of Theorem \ref{thrm2}}

The martingale $\{V_t^x(\lambda) \}$ will play an important role in the proof of the following result.

\begin{prop}\label{lemma_deri_conver}
	Suppose that $\lambda\geq \lambda^*$. Then $\partial W(\lambda,y) = \lim_{t\uparrow\infty} \partial W_t(\lambda)$ exists $\P_y$-almost surely in $[0,\infty)$. Furthermore, $\P_y(\partial W(\lambda,y)=0) = 0$ or $1$.
\end{prop}
\begin{proof}
Let $x\in \R$ be such that $y>h^{-1}(-x)$.
Since $V_t^x(\lambda)$ is a non-negative martingale, it has an almost sure limit. Let $\gamma^{(-x,\lambda)}$ denote the event that the BBMPE remains entirely to the right of
$\mathbf{\Gamma}^{(-x,\lambda)}$.
On this event we have $V_t^x(\lambda) = \partial W_t(\lambda) + xW_t(\lambda)$. Hence, on $\gamma^{(-x,\lambda)}$, $\lim_{t\uparrow\infty}(\partial W_t(\lambda) + xW_t(\lambda))$ exists and equals $\lim_{t\uparrow\infty} V_t^x(\lambda)\geq 0$. Note that when $\lambda\geq \lambda^*$ we have $W(\lambda,y) = 0$
$\P_y$-almost surely.
Therefore, we have $\lim_{t\uparrow\infty} V_t^x(\lambda) = \lim_{t\uparrow\infty} \partial W_t(\lambda)$ on $\gamma^{(-x,\lambda)}$.

Let $m_t := \min\{X_u(t): u\in N_t \}$. Thus
\begin{align*}
W_t(\lambda) = e^{-\gamma(\lambda)t} \sum_{u\in N_t} e^{-\lambda X_u(t)} \psi(X_u(t),\lambda) \geq c
e^{-\gamma(\lambda)t-\lambda m_t},
\end{align*}
where $c=\min_{z\in[0,1]} \psi(z,\lambda)$ is a positive constant. Since $W(\lambda^*,y)=0$, we have\\ $\lim_{t\uparrow\infty}
e^{-\gamma(\lambda^*)t-\lambda^* m_t} = 0$,
that is, $\lim_{t\uparrow\infty} m_t + \gamma'(\lambda^*)t = \infty$.
Hence $\inf_{t\geq 0} \{m_t+\gamma'(\lambda)t \} > -\infty$
$\P_y$-almost surely for all $\lambda\geq \lambda^*$. Therefore
\begin{equation*}
\P_y(\gamma^{(-x,\lambda)}) \geq \P_y\left(\inf_{t\geq 0}
\{m_t+\gamma'(\lambda)t \} >
-x + \max_{z\in[0,1]} \frac{\psi_{\lambda}(z,\lambda)}{\psi(z,\lambda)} \right) \uparrow 1 \quad \text{ as } x\uparrow\infty.
\end{equation*}
Now the existence of an almost sure limit for the derivative martingale in $[0,\infty)$ has been established.

It remains to prove that $\partial W(\lambda,y)$ is either strictly positive or zero with probability one.
Noticing that
\begin{align*}\partial W_t(\lambda) = &e^{-\gamma(\lambda)s} \sum_{u\in N_s} e^{-\gamma(\lambda)(t-s)} \sum_{v\in N_t, v>u} e^{-\lambda X_v(t)} \\
&\times\left(\psi(X_u(t),\lambda)(\gamma'(\lambda)(t-s)+X_u(t)) - \psi_{\lambda}(X_u(t),\lambda) + \gamma'(\lambda) s \psi(X_u(t),\lambda) \right),
\end{align*}
we have under $\P_y$,
\begin{equation}\label{decomp_deriv_t}
\partial W_t(\lambda) \overset{d}{=} e^{-\gamma(\lambda)s} \sum_{u\in N_s} \left( \partial W_{t-s}^{(u)}(\lambda,X_u(s)) + \gamma'(\lambda)s W_{t-s}^{(u)}(\lambda,X_u(s)) \right),
\end{equation}
where $W_{t-s}^{(u)}(\lambda, X_u(s))$ and $\partial W_{t-s}^{(u)}(\lambda, X_u(s))$ are the
additive and derivative martingales for the BBMPE starting from $X_u(s)$,
and given $\F_s$, $\{(W_{t-s}^{(u)}(\lambda, X_u(s)), \partial W_{t-s}^{(u)}(\lambda, X_u(s))): u\in N_s\}$ are independent.
Letting $t\rightarrow\infty$ and noticing that $W(\lambda,y) = 0$ $\P_y$-almost surely for $\lambda\geq \lambda^*$, we have under $\P_y$
\begin{equation}\label{decomp_deriv}
\partial W(\lambda,y) \overset{d}{=} e^{-\gamma(\lambda)s} \sum_{u\in N_s} \partial W^{(u)}(\lambda,X_u(s)),
\end{equation}
where $\partial W^{(u)}(\lambda,X_u(s))$ is the limit of the derivative martingale for the BBMPE starting from $X_u(s)$.
Define
\begin{equation}\label{def_px_deri}
p(y) := \P_y (\partial W(\lambda,y) = 0),\quad y\in \R.
\end{equation}
Using the 1-periodicity of  $\mathbf g(\cdot)$ and $\psi(\cdot,\lambda)$, and the fact $W(\lambda,y) = 0$ $\P_y$-almost surely for $\lambda\geq \lambda^*$, we get that for $y-z\in\mathbb{Z}$,
\begin{equation}\label{mart_deriv_periodic}
(\partial W(\lambda,y), \P_y) \overset{d}{=} (e^{-\lambda(y-z)} \partial W(\lambda,z), \P_z).
\end{equation}
Thus, it suffices to verify $p(y) = 0$ or $1$ for any $y\in [0,1)$. For any $z\in [0,1)$, we define the following two stopping times with respect to BBMPE:
\begin{align*}
\tau_z^B &:= \inf\{t\geq 0: B_t = z \},\\
\tau_z^Z &:= \inf\{t\geq 0: \exists u\in N_t \; s.t. \; X_u(t) = z \}.
\end{align*}
At  each fission time, there is at least one offspring. As a result, for any $M\geq 0$ and $y\in [0,1)$,
\begin{equation*}
\P_y (\tau^Z_z \leq M) \geq  \Pi_y (\tau_z^B \leq M).
\end{equation*}
By the decomposition \eqref{decomp_deriv} and the strong Markov property, we have
\begin{equation}
p(y) \geq p(z) \P_y (\tau^Z_z \leq M) \geq p(z) \Pi_y (\tau_z^B \leq M).
\end{equation}
Since  $\Pi_y(\tau_z^B<\infty)=1$, letting $M\rightarrow\infty$, we obtain $p(y) \geq p(z)$. By symmetry, we get
\begin{align}
\P_y (\partial W(\lambda,y) = 0) = \P_z (\partial W(\lambda,z) = 0),\quad \forall y,z\in\R.
\end{align}
Let $p$ denote this common value. It follows from
\eqref{decomp_deriv} that $p$ satisfies $p=\E_y p^{|N_s|}$ for any $s>0$.
Since $|N(s)|\geq 1$ almost surely and $\P_y(|N(s)|>1)>0$,  $p=\E_y p^{|N_s|}$
implies $p=0$ or $1$. This completes the proof.
\end{proof}

\subsection{Measure change by $V$}
We have shown in Lemma \ref{lemma_Vt} that,
for any $x, y\in \R$ with $y>h^{-1}(-x)$, $\{V_t^x(\lambda), t\geq 0; \P_y\}$ is a martingale.
We now assume $x, y\in \R$ with $y>h^{-1}(-x)$ and use $V_t^x(\lambda)$
to define a probability measure $\P_y^{(x,\lambda)}$ on $(\mathcal{T}. \mathcal{F})$:
\begin{equation}\label{meas_change_V}
\frac{\mathrm{d}\P_y^{(x,\lambda)}}{\mathrm{d}\P_y}\bigg{|} _{\mathcal{F}_t} = \frac{V_t^x(\lambda)}{V_0^x(\lambda)}.
\end{equation}
According to
\cite{HH09} or \cite{RS},
there exists a probability measure $\widetilde{\P}_y^{(x,\lambda)}$ on $(\widetilde{\mathcal{T}}, \widetilde{\F})$ such that
\begin{equation}\label{V-tildeF}
\P_y^{(x,\lambda)} = \widetilde{\P}_y^{(x,\lambda)}|_{\F},
\end{equation}
 and under $\widetilde{\P}_y^{(x,\lambda)}$:
\begin{enumerate}[(i)]
	\item the ancestor starts from $y$  and the spine $\xi$ moves according
    to $\Pi_y^{(x,\lambda)}$, which implies that
    \begin{equation}\label{e:R_t}
    R_t := -h(y)+\gamma'(\lambda)T(t) + h(X_{\xi}(T(t)) ),
    \end{equation}	
    then $\left\{x+h(y)+R_t : \, t\geq 0 \right\}$
    is a standard Bessel-3 process started at $x+h(y)$,
    where $T(t) = \inf\{s\geq 0: \langle M \rangle_s > t \}$ and $\langle M \rangle_t =
    \int_0^t [h'(X_{\xi}(s))]^2 \mathrm{d}s$;
 	\item given the trajectory $X_{\xi}$ of the spine, the branching rate is given by $(m+1)\mathbf g(X_{\xi}(t))$;
	\item at the fission time of node $v$ on the spine, the single spine particle is replaced by $1+A_v$ offspring, with $A_v$ being independent identically distributed as $\{\tilde{p}_k:k\geq 0 \}$, where $\tilde{p}_k = \frac{(k+1)p_k}{m+1}$;
	\item the spine is chosen uniformly from the $1+A_v$ offspring at the fission time of $v$;
	\item the remaining $A_v$ particles $O_v$ give rise to the independent subtrees $\{(\tau,B,\eta)_j^v\}$, $j\in O_v$, and they evolve as
	independent processes determined by the measure $\P_{X_v(d_v)}$ shifted to their point
	and time of creation.
\end{enumerate}

\subsection{Proof of Theorem \ref{thrm2}}
Assume $x, y\in \R$ satisfy $y>h^{-1}(-x)$.
Let $V^x(\lambda) = \limsup_{t\uparrow\infty} V_t^x(\lambda)$ and using the fundamental measure theoretic result again, we have
\begin{align}
V^x(\lambda) = \infty, \;
\widetilde{\P}_y^{(x,\lambda)}\text{-a.s.}
\quad &\Longleftrightarrow \quad V^x(\lambda) = 0, \; \P_y\text{-a.s.}
\label{e:new1}\\
V^x(\lambda) < \infty, \;
\widetilde{\P}_y^{(x,\lambda)}\text{-a.s.}
\quad &\Longleftrightarrow \quad \int V^x(\lambda) \mathrm{d}\P_y =1.
\label{e:new2}
\end{align}

\begin{thrm}\label{thrm_V}
For $x,y\in\R$ satisfying $y>h^{-1}(-x)$,
the almost sure limit $V^x(\lambda)$ has the following properties:
\begin{enumerate}[(i)]
	\item If $\lambda > \lambda^*$ then $V^x(\lambda) = 0$ $\P_y$-almost surely.
	\item If $\lambda = \lambda^*$ then $V^x(\lambda) = 0$ $\P_y$-almost surely or is an $L^1(\P_y)$-limit according to $\mathbf E(L(\log^+L)^2) = \infty$ or $\mathbf E(L(\log^+L)^2) < \infty$.
	\item If $\lambda \in [0,\lambda^*)$ then $V^x(\lambda)=0$, $\P_y$-almost surely or is an $L^1(\P_y)$-limit according to $\mathbf E(L\log^+L) = \infty$ or $\mathbf E(L\log^+L) < \infty$.
\end{enumerate}
\end{thrm}
\begin{proof} Recall that $\tilde{p}_k = (k+1)p_k/(m+1)$.
Suppose $q>0$. A simple calculation shows that, for any fixed $c>0$,
$\mathbf E(L(\log^+L)^q)<\infty$
 if and only if
\begin{equation*}
\sum_{n\geq 1}  \widetilde{\mathbf P} (\log L > cn^{1/q})<\infty,
\end{equation*}
where under $\widetilde{\mathbf P}$,
$L$ has distribution $\{\tilde{p}_k:k\geq 0 \}$. Therefore, if $\{A_n: n\geq 0 \}$ is a sequence of independent copies of $L$ under $\widetilde{\mathbf P}$, then by the Borel-Cantelli lemma,
\begin{align}
\underset{n \rightarrow \infty}{\limsup } \; n^{-1/q} \log A_{n}=\left\{\begin{array}{ll}
0 & \text { if } \mathbf E\left(L (\log^+ L)^q\right)<\infty, \\
\infty & \text { if } \mathbf E\left(L (\log^+ L)^q \right)=\infty
\end{array}\right.
\end{align}
$\widetilde{\mathbf P}$-almost surely.

($\mathrm{i}$) Suppose that $\lambda>\lambda^*$. By Lemma \ref{lemma_com}, $\gamma'(\lambda) > \gamma(\lambda)/\lambda$.
Then by the definition of $V_t^x(\lambda)$ in \eqref{mart_V},
\begin{align}
V_{T(t)}^x(\lambda)
\geq &e^{-\lambda\left(X_{\xi}(T(t)) + \frac{\gamma(\lambda)}{\lambda}T(t)\right)} \psi(X_{\xi}(T(t)),\lambda)
\left( x+\gamma'(\lambda)T(t)+h(X_{\xi}(T(t))) \right)\\
=& e^{-\lambda\left(X_{\xi}(T(t)) + \gamma'(\lambda)T(t)\right) + \lambda\left(\gamma'(\lambda)-\frac{\gamma(\lambda)}{\lambda}\right)T(t)} \psi(X_{\xi}(T(t)),\lambda)
\left(x+h(y)+R_t\right)\\
\geq &c(\lambda)
e^{-\lambda R_t + \lambda\left(\gamma'(\lambda)-\frac{\gamma(\lambda)}{\lambda}\right)T(t)}
\left(x+h(y)+R_t\right),
\end{align}
where the constant $c(\lambda) := e^{-\lambda h(y)} \inf_{x\in\R} \{e^{-\lambda(x-h(x))} \psi(x,\lambda) \}$.
Under $\widetilde{\P}_y^{(x,\lambda)}$, $\{x+h(y)+R_t\}$
is a Bessel-3 process, and so for any $\epsilon>0$, this process eventually grows no faster than $t^{1/2+\epsilon}$ and no slower than $t^{1/2-\epsilon}$.
By Lemma \ref{lemma_mart_h}, there exist two positive constants $c_1\leq c_2$ such that $\langle M\rangle_t \in [c_1t, c_2 t]$ and hence $\frac{t}{c_2}\leq T(t) \leq \frac{t}{c_1}$.
Combining these with $\gamma'(\lambda) > \gamma(\lambda)/\lambda$, we get that
\begin{align*}
V^x(\lambda) = \limsup_{t\uparrow\infty} V_{T(t)}^x(\lambda)
\ge c(\lambda)
\limsup_{t\uparrow\infty} e^{-c\lambda t^{1/2+\epsilon} + \lambda(\gamma'(\lambda)-\frac{\gamma(\lambda)}{\lambda})\frac{t}{c_2} } t^{1/2-\epsilon} = \infty , \quad
\widetilde{\P}_y^{(x,\lambda)}\text{-a.s.}
\end{align*}
Hence, by \eqref{e:new1}, $V^x(\lambda)=0$,  $\P_y$-almost surely.

($\mathrm{ii}$) Suppose that $\lambda=\lambda^*$
which, by Lemma \ref{lemma_com}, implies that $\gamma'(\lambda^*) = \gamma(\lambda^*)/\lambda^*$.
We first consider the case that $\mathbf E(L(\log^+L)^2) = \infty$.
Recall that $d_{\xi_k}$ is the death time of the particle $\xi_k$ on the spine and $1+A_{\xi_k}$
is the number of its offspring. We have
$$
V_{d_{\xi_n}}^x(\lambda^*)
\geq
A_{\xi_n} \left( x+\gamma'(\lambda^*)d_{\xi_n} +
h(X_{\xi}(d_{\xi_n})) \right) e^{-\gamma(\lambda^*)d_{\xi_n}-\lambda^*X_{\xi}(d_{\xi_n})} \psi(X_{\xi}(d_{\xi_n}),\lambda^*).
$$
We only need to prove that, $\widetilde{\P}_y^{(x,\lambda)}$-almost surely,
$$
\limsup_{n\rightarrow\infty} A_{\xi_n} \left( x+\gamma'(\lambda^*)d_{\xi_n} +
h(X_{\xi}(d_{\xi_n}))
\right) e^{-\gamma(\lambda^*)d_{\xi_n}-\lambda^*X_{\xi}(d_{\xi_n})} \psi(X_{\xi}(d_{\xi_n}),\lambda^*)
= +\infty.
$$
Define $v_n$ such that $T(v_n) = d_{\xi_n}$, that is, $\langle M \rangle_{d_{\xi_n}} = v_n$. Then
\begin{equation*}
x +h(y) + R_{v_n} = x+\gamma'(\lambda^*)d_{\xi_n} + h(X_{\xi}(d_{\xi_n})).
\end{equation*}
It suffices to show that for any $M>0$,
    $\widetilde{\P}_y^{(x,\lambda)}$-almost surely,
\begin{equation}
\sum_{n=0}^{+\infty} \mathbf{1}_{\left\{ A_{\xi_n}
(x +h(y)+ R_{v_n})
 e^{-\lambda^*\left( X_{\xi}(T(v_n))+\gamma'(\lambda^*)T(v_n) \right)} \psi(X_{\xi}(T(v_n)),\lambda^*) \geq M \right\}}
 = +\infty .
\end{equation}
Since $\inf_{z\in\R} \psi(z,\lambda^*)>0$, it suffices to show that for any $M>0$,
\begin{equation}\label{equ_sum_infty}
\sum_{n=0}^{+\infty} \mathbf{1}_{\left\{ A_{\xi_n}
(x + h(y) + R_{v_n})
e^{-\lambda^* R_{v_n}} \geq M \right\}} = +\infty \quad \widetilde{\P}_y^{(x,\lambda)}\text{-a.s.}
\end{equation}

Let $\mathcal{G}$ denote the $\sigma$-field generated by $X_{\xi}$ (the spatial path of the spine). For any set $B\in \mathcal{B}[0,+\infty)\times \mathcal{B}(\mathbb{Z}_+)$, define
\begin{equation}\label{PRM_varphi}
\varphi(B) := \# \{n:(v_n, A_{\xi_n})\in B \}.
\end{equation}
We first show that, conditioned on $\mathcal{G}$, $\varphi$ is a Poisson random measure on $[0,+\infty)\times \mathbb{Z}_+$ with intensity
$(m+1)\mathbf g(X_{\xi}(T(t)))\mathrm{d}T(t) \sum_{k\in\mathbb{Z}_+} \tilde{p}_k \delta_k(\mathrm{d}y)$.
For simplicity, given $\mathcal{G}$, put $f(t) = \langle M \rangle_t = \int_0^t \left[ h'(X_{\xi}(s)) \right]^2 \mathrm{d}s$.
Then, it is known that $f(t)$ is a strictly increasing $C^1$-function and $f'(t)\in [c_1,c_2]$. Hence $T(t) = f^{-1}(t)$ and $T'(t) \in [1/c_2, 1/c_1]$.
Define
\begin{equation}
\widetilde{\varphi}(B) := \# \{n:(d_{\xi_n}, A_{\xi_n})\in B \}.
\end{equation}
Using the spine decomposition, it is easy to show that, conditioned on $\mathcal{G}$, $\widetilde{\varphi}$ is a Poisson random measure on $[0,+\infty)\times \mathbb{Z}_+$ with intensity $(m+1)\mathbf g(X_{\xi}(t))\mathrm{d}t \sum_{k\in\mathbb{Z}_+} \tilde{p}_k \delta_k(\mathrm{d}y)$. Note that, given the spatial path of the spine, $f(t)$ is a deterministic increasing function and $v_n = f(d_{\xi_n})$. It is not difficult to verify $\varphi$ satisfies the definition of Poisson random measure. Moreover, for any
$D\in \mathcal{B}(Z_+)$, $\widetilde{\varphi}([0,t]\times D) = \varphi([0,f(t)]\times D)$.
By making the change of variables $s=T(u)$, we have
\begin{equation}
\int_0^t (m+1)\mathbf g(X_{\xi}(s)) \mathrm{d}s = \int_0^{f(t)} (m+1) \mathbf g(X_{\xi}(T(u))) \mathrm{d}T(u).
\end{equation}
Hence, conditioned on $\mathcal{G}$,
the intensity of $\varphi$ is
$$
(m+1)\mathbf g(X_{\xi}(T(t)))\mathrm{d}T(t) \sum_{k\in\mathbb{Z}_+} \tilde{p}_k \delta_k(\mathrm{d}y).
$$
Thus for any $t\in(0,+\infty)$, given $\mathcal{G}$,
\begin{equation}
N_t:=\#\left\{n:v_n\leq t, A_{\xi_n}
(x +h(y) + R_{v_n})
e^{-\lambda^* R_{v_n}} \geq M   \right\}
\end{equation}
is a Poisson random variable with parameter
\begin{equation}
\int_0^t (m+1)\mathbf g(X_{\xi}(T(s))) \sum_k \tilde{p}_k
\mathbf{1}_{\left\{k(x+h(y) + R_s)
e^{-\lambda^* R_s} \geq M \right\} } \mathrm{d}T(s).
\end{equation}
Since $\min_{z\in\R} \mathbf g(z) >0$ and $T'(t) \in [1/c_2,1/c_1]$,
to prove \eqref{equ_sum_infty}, it suffices to show that
\begin{align}\label{int=infty}
\int_0^{+\infty} (m+1) \sum_k \tilde{p}_k
\mathbf{1}_{\left\{k\left(x +h(y) + R_t\right)
 e^{-\lambda^* R_t} \geq M \right\} } \mathrm{d}t = +\infty,  \quad \widetilde{\P}_y^{(x,\lambda)}\text{-a.s.}
\end{align}
For any $c\in (0,+\infty)$,
put
\begin{equation}\label{def-Ac}
A_c:=\left\{\int_0^{+\infty} (m+1) \sum_k \tilde{p}_k
\mathbf{1}_{\left\{k\left(x +h(y) + R_t\right)
 e^{-\lambda^* R_t} \geq M \right\} } \mathrm{d}t  \leq c\right\}.
\end{equation}
Using arguments similar to those in the proof of
\cite[Theorem 1]{YR},
we get that $\widetilde{\P}_y^{(x,\lambda)}(A_c) = 0$  (see Lemma \ref{Prob-Ac} in the Appendix for a proof), which implies \eqref{int=infty} holds.
Therefore, we have $V^x(\lambda) = \infty, \; \widetilde{\P}_y^{(x,\lambda)}\text{-a.s.}$ Hence $V^x(\lambda) = 0$, $\P_y$-a.s.

Now we consider the case that $\mathbf E(L(\log^+L)^2) < \infty$.
Let $\widetilde{\mathcal{G}}$ be the $\sigma$-field generated by the motion of the spine and the genealogy along the spine
($\mathcal{G}\subset \widetilde{\mathcal{G}}$).
Recall that $\{n_t: t\geq 0 \}$ is the counting process of fission points along the spine. Using the spine decomposition and the martingale property of $V_t^x(\lambda^*)$, we have
\begin{align*}
&\widetilde{\P}_y^{(x,\lambda)}  (V_t^x(\lambda^*) | \widetilde{\mathcal{G}}) = \left(x+\gamma'(\lambda^*)t +
X_{\xi}(t) - \frac{\psi_{\lambda}(X_{\xi}(t),\lambda^*)}{\psi(X_{\xi}(t),\lambda^*)}
\right) e^{-\lambda^*X_{\xi}(t)-\gamma(\lambda^*)t}  \psi(X_{\xi}(t),\lambda^*)\\ &+\sum_{k=0}^{n_t-1} A_{\xi_{k}}\left(x+\gamma'(\lambda^*)d_{\xi_{k}} + X_{\xi}(d_{\xi_{k}}) - \frac{\psi_{\lambda}(X_{\xi}(d_{\xi_{k}}),\lambda^*)}{\psi(X_{\xi}(d_{\xi_{k}}),\lambda^*)} \right) e^{-\lambda^*X_{\xi}(d_{\xi_{k}})-\gamma(\lambda^*)d_{\xi_{k}}}  \psi(X_{\xi}(d_{\xi_{k}}),\lambda^*).
\end{align*}
Next we show that $\widetilde{\P}_y^{(x,\lambda)}$-almost surely,
\begin{align}\label{YR_5}
&\sum_{n=0}^{+\infty} A_{\xi_{n}}\left(x+\gamma'(\lambda^*)d_{\xi_{n}} + X_{\xi}(d_{\xi_{n}}) - \frac{\psi_{\lambda}(X_{\xi}(d_{\xi_{n}}),\lambda^*)}{\psi(X_{\xi}(d_{\xi_{n}}),\lambda^*)} \right) e^{-\lambda^*X_{\xi}(d_{\xi_{n}})-\gamma(\lambda^*)d_{\xi_{n}}}  \psi(X_{\xi}(d_{\xi_{n}}),\lambda^*)\\
& < +\infty.
\end{align}
Using an argument similar to the one above, it is equivalent to prove
\begin{equation}
\sum_{n=0}^{+\infty} A_{\xi_{n}} \left(x - \frac{\phi_{\lambda}(y,\lambda^*)}{\phi(y,\lambda^*)} + R_{v_n}\right) e^{-\lambda^* R_{v_n}} < +\infty, \quad \widetilde{\P}_y^{(x,\lambda)}\text{-a.s.}
\end{equation}
For simplicity, we use $\tilde x$
to denote $x - \frac{\phi_{\lambda}(y,\lambda^*)}{\phi(y,\lambda^*)}$ and we will show
\begin{equation}
\sum_{n=0}^{+\infty} A_{\xi_{n}} \left(\tilde x + R_{v_n}\right) e^{-\lambda^* R_{v_n}} < +\infty \quad \widetilde{\P}_y^{(x,\lambda)}\text{-a.s.}
\end{equation}
Choose any $h\in (0,\lambda^*)$,
\begin{align}\label{YR_15}
\sum_{n=0}^{+\infty} A_{\xi_{n}} \left(\tilde x + R_{v_n}\right) e^{-\lambda^* R_{v_n}} = &\sum_{n=0}^{+\infty} A_{\xi_{n}} \left(\tilde x + R_{v_n}\right) e^{-\lambda^* R_{v_n}} \textbf{1}_{\{A_{\xi_{n}}\leq e^{hR_{v_n}}\}} \\
&+ \sum_{n=0}^{+\infty} A_{\xi_{n}} \left(\tilde x + R_{v_n}\right) e^{-\lambda^* R_{v_n}} \textbf{1}_{\{A_{\xi_{n}}> e^{hR_{v_n}}\}} \notag\\
=&: \mathrm{I} + \mathrm{II}.
\end{align}
We will prove that both I and II are finite $\widetilde{\P}_y^{(x,\lambda)}$-almost surely.

Recall that $\varphi$ is defined by \eqref{PRM_varphi}. We rewrite I as
\begin{equation*}
\mathrm{I} = \int_{[0,+\infty)\times \mathbb{Z}_+} (\tilde x+R_s)y e^{-\lambda^* R_s} \mathbf{1}_{\{y\leq e^{hR_s} \}} \varphi(\mathrm{d}s \times \mathrm{d}y).
\end{equation*}
Since
$\widetilde{\P}_y^{(x,\lambda)}(\mathrm{I}) = \widetilde{\P}_y^{(x,\lambda)}(\widetilde{\P}_y^{(x,\lambda)}(I|\mathcal{G}))$,
by the compensation formula of Poisson random measures,
\begin{align*}
\widetilde{\P}_y^{(x,\lambda)}(\mathrm{I}) &= \widetilde{\P}_y^{(x,\lambda)}\left(\int_0^{+\infty} (m+1)\mathbf g(X_{\xi}(T(s)))(\tilde x+R_s) e^{-\lambda^*R_s} \sum_k \tilde{p}_k k \mathbf{1}_{\{k\leq e^{hR_s} \}} \mathrm{d}T(s) \right)\\
&\lesssim \widetilde{\P}_y^{(x,\lambda)}\left(\int_0^{+\infty} (\tilde x+R_s) e^{-\lambda^*R_s} \sum_k \tilde{p}_k k \mathbf{1}_{\{k\leq e^{hR_s} \}} \mathrm{d}s \right)\\
&\leq  \sum_k \tilde{p}_k \int_0^{+\infty} \widetilde{\P}_y^{(x,\lambda)}\left( (\tilde x+R_s) e^{-(\lambda^*-h)R_s} \mathbf{1}_{\{R_s\geq h^{-1}\log^+k \}} \right) \mathrm{d}s .
\end{align*}
In the display above and also in the sequel,
we write $A\lesssim B$
when there exists a constant $c>0$,  such that $A\leq cB$.
Under $\widetilde{\P}_y^{(x,\lambda)}$, $\tilde x+R_s$ is a Bessel-3 process, which has the same distribution as
$|W_t+\hat{x}|$ under $\mathbf{P}_w$,
where $(W_t,\mathbf{P}_w)$
is a 3-dimensional standard Brownian motion starting from 0 and $\hat{x}$ is a point in $\mathbb{R}^{3}$ with norm $\tilde x$. Thus
\begin{align*}
\widetilde{\P}_y^{(x,\lambda)}(\mathrm{I}) & \lesssim \sum_{k} \tilde{p}_{k} \int_{0}^{+\infty}
\mathbf{P}_w\left(\left|W_{s}+\hat{x}\right| e^{-(\lambda^*-h)\left|W_{s}+\hat{x}\right|} \mathbf{1}_{\left\{\left|W_{s}+\hat{x}\right| \geq h^{-1} \log ^{+} k+\tilde x\right\}}\right) \mathrm{d} s \\
& \lesssim \sum_{k} \tilde{p}_{k} \int_{\left\{|y+\hat{x}| \geq h^{-1} \log ^{+} k+\tilde x\right\}}|y+\hat{x}| e^{-(\lambda^*-h)|y+\hat{x}|} \mathrm{d} y \int_{0}^{+\infty} s^{-3 / 2} e^{-|y|^{2} / 2 \pi s} \mathrm{~d} s \\
&=\sum_{k} \tilde{p}_{k} \int_{\left\{|y+\hat{x}| \geq h^{-1} \log ^{+} k+\tilde x\right\}} \frac{|y+\hat{x}|}{|y|} e^{-(\lambda^*-h)|y+\hat{x}|} \mathrm{d} y \int_{0}^{+\infty} t^{-1 / 2} e^{-t / 2 \pi} \mathrm{d} t \\
& \lesssim \sum_{k} \tilde{p}_{k} \int_{\left\{|y+\hat{x}| \geq h^{-1} \log ^{+} k+\tilde x\right\}} \frac{|y+\hat{x}|}{|y|} e^{-(\lambda^*-h)|y+\hat{x}|} \mathrm{d} y \\
& \leq \sum_{k} \tilde{p}_{k} \int_{\left\{|y| \geq h^{-1} \log ^{+} k\right\}} \frac{|y|+\tilde x}{|y|} e^{-(\lambda^*-h)(|y|-\tilde x)} \mathrm{d} y.
\end{align*}
Using spherical coordinates in the last integral, we get
\begin{align}
\widetilde{\P}_y^{(x,\lambda)}(\mathrm{I}) &\lesssim \sum_{k} \tilde{p}_{k} \int_{h^{-1}\log^+k}^{+\infty} (r^2+\tilde xr) e^{-(\lambda^*-h)r} \mathrm{d}r < +\infty,
\end{align}
and therefore, $\widetilde{\P}_y^{(x,\lambda)}(\mathrm{I}<+\infty) = 1$.

On the other hand, similar calculation yields
\begin{align}
	&\widetilde{\P}_y^{(x,\lambda)} \left(\sum_{n=0}^{+\infty} \mathbf{1}_{\left\{A_{\xi_n}>e^{hR_{v_n}}\right\}}\right) \\
    =&(1+m) \sum_{k} \tilde{p}_{k} \widetilde{\P}_y^{(x,\lambda)} \left( \int_{0}^{+\infty}  \mathbf g(X_{\xi}(T(s))) \mathbf{1}_{\{\tilde x+R_s<h^{-1}\log^+ k + \tilde x \} }  \mathrm{d} T(s) \right)\\
	\lesssim &\sum_{k} \tilde{p}_{k} \int_{0}^{+\infty}
    \mathbf{P}_w\left(\left|W_{s}+\hat{x}\right|<h^{-1} \log ^{+} k+\tilde x\right) \mathrm{d} s \\
	\lesssim &\sum_{k} \tilde{p}_{k} \int_{\left(|y+\hat{x}|<h^{-1} \log ^{+} k+\tilde x\right\}} \mathrm{d} y \int_{0}^{+\infty} s^{-3 / 2} \mathrm{e}^{-\left|y\right|^{2} / 2 \pi s} \mathrm{d} s \\
	\lesssim &\sum_{k} \tilde{p}_{k} \int_{\left\{|y| \leq h^{-1} \log ^{+} k+2 \tilde x\right\}}|y|^{-1} \mathrm{~d} y \\
	\lesssim &\sum_{k} \tilde{p}_{k}\left(h^{-1} \log ^{+} k+2 \tilde x\right)^{2}.
\end{align}
The assumption that $\mathbf E(L(\log^+L)^2)<+\infty$ implies that $\sum_{k\in\mathbb{Z}_+} \tilde{p}_k (\log^+k)^2 < +\infty$, which implies that the right side of the last inequality is finite. Hence, $\sum_{n=0}^{+\infty} \mathbf{1}_{\left\{A_{\xi_n}>e^{hR_{v_n}}\right\}} < +\infty$, $\widetilde{\P}_y^{(x,\lambda)}$-almost surely,
that is, II is a sum of finitely many terms.
It follows that $\widetilde{\P}_y^{(x,\lambda)}(\mathrm{II}<+\infty) = 1$. Hence \eqref{YR_5} is valid, which implies that
\begin{equation}
\limsup_{t\uparrow\infty}  \widetilde{\P}_y^{(x,\lambda)} (V_t^x(\lambda^*)\, |\, \widetilde{\mathcal{G}}) < \infty \quad \widetilde{\P}_y^{(x,\lambda)}\text{-a.s.}
\end{equation}
By Fatou's lemma, $\liminf_{t \uparrow\infty} V_{t}^{x}(\lambda^*)<\infty$,  $\widetilde{\P}_y^{(x,\lambda)}$-a.s.
The Radon-Nikodym derivative \eqref{meas_change_V} and \eqref{V-tildeF} tells us that  $V_{t}^{x}(\lambda^*)^{-1}$ is a $\widetilde\P_y^{(x,\lambda)}$-supermartingale and therefore has a limit $\widetilde\P_y^{(x,\lambda)}$-almost surely.
It follows that
$$\limsup _{t \uparrow \infty} V_{t}^{x}(\lambda^*)=\liminf _{t \uparrow \infty} V_{t}^{x}(\lambda^*)<\infty,\quad \widetilde{\P}_y^{(x,\lambda)}\mbox{-a.s.}$$
Hence, by \eqref{e:new2},
$V^x(\lambda^*)$ is an $L^{1}(\P_y)$ limit when $\mathbf E(L(\log^+L)^2)<\infty$.

(iii) Now suppose $\lambda\in[0,\lambda^*)$ and $\mathbf E(L\log^+L) = \infty$.
By Lemma \ref{lemma_com}, $\gamma'(\lambda)<\gamma(\lambda)/\lambda$.
We have
\begin{align*}
V_{d_{\xi_n}}^x(\lambda) &\geq A_{\xi_n} \left( x+\gamma'(\lambda)d_{\xi_n} + h(X_{\xi}(d_{\xi_n}))
\right) e^{-\gamma(\lambda)d_{\xi_n}-\lambda X_{\xi}(d_{\xi_n})} \psi(X_{\xi}(d_{\xi_k}),\lambda)\\
&\gtrsim A_{\xi_n} \left( x - \frac{\phi_{\lambda}(y,\lambda)}{\phi(y,\lambda)} + R_{v_n} \right) e^{-\lambda R_{v_n} - \lambda(\frac{\gamma(\lambda)}{\lambda}-\gamma'(\lambda)) T(v_n)}.
\end{align*}
Using argument similar to part (ii), we have $V^x(\lambda) = \infty$, $\P_y^{(x,\lambda)}$-a.s. and hence $V^x(\lambda) = 0$ $\P_y$-a.s. We omit the details here.

When $\lambda\in[0,\lambda^*)$ and $\mathbf E(L\log^+L) < \infty$, using the spine decomposition, we have
\begin{align*}
&\widetilde{\P}_y^{(x,\lambda)}  (V_t^x(\lambda) | \widetilde{\mathcal{G}})\\
=& \left(x+\gamma'(\lambda)t + h(X_{\xi}(t)) \right) e^{-\lambda(X_{\xi}(t)+\gamma'(\lambda)t)} e^{-\lambda(\frac{\gamma(\lambda)}{\lambda}-\gamma'(\lambda)) t} \psi(X_{\xi}(t),\lambda)
\\ &+\sum_{k=0}^{n_t-1} A_{\xi_{k}}\left(x+\gamma'(\lambda)d_{\xi_{k}} +
h(X_{\xi}(d_{\xi_{k}}))\right)e^{-\lambda(X_{\xi}(d_{\xi_{k}})+\gamma'(\lambda)d_{\xi_{k}})-\lambda(\frac{\gamma(\lambda)}{\lambda}-\gamma'(\lambda)) d_{\xi_{k}}}\psi(X_{\xi}(d_{\xi_{k}}),\lambda).
\end{align*}
Using argument similar to  part (ii), we obtain $V_x(\lambda)$ is an $L^1(\P_y)$-limit. We omit the details here.
\end{proof}

\begin{proof}[Proof of Theorem \ref{thrm2}]
Suppose $\lambda\geq \lambda^*$.
The case $\lambda\leq -\lambda^*$ follows by symmetry.
For a given $x\in \R$, let $y\in \R$ be such that $y>h^{-1}(-x)$. By Proposition \ref{lemma_deri_conver}, we know that under $\P_y$, $V^x(\lambda) = \partial W(\lambda,y)$ on the event $\gamma^{(-x,\lambda)}$. And also
\begin{equation}
\P_y(\gamma^{(-x,\lambda)}) \rightarrow 1 \quad \mbox{as } x\rightarrow\infty.
\end{equation}
Combining these with Theorem \ref{thrm_V}, we get  $\partial W(\lambda,y) = 0$ $\P_y$-almost surely when $V^x(\lambda) = 0$ $\P_y$-almost surely.
It follows from Proposition \ref{lemma_deri_conver} that $\P_y(\partial W(\lambda,y)=0) = 0$ or $1$. Therefore, when $V^x(\lambda)$ is an $L^1(\P_y)$-limit, we have $\partial W(\lambda,y) \in (0,\infty)$. So Theorem \ref{thrm_V} implies Theorem \ref{thrm2}.
\end{proof}

\section{Proof of Theorem \ref{thrm3}}

It was proved analytically in \cite[Theorem 1.2]{BH} and \cite[Proposition 1.2]{H08} that pulsating travelling waves
exist if and only if $|\nu|\geq \nu^*$.
In this section, we will use probabilistic methods to prove the existence in the supercritical case $|\nu|>\nu^*$ and critical case $|\nu|=\nu^*$, and the nonexistence in the subcritical case $|\nu|<\nu^*$.

\subsection{Existence in the supercritical case ($\nu>\nu^*$)}\label{section_exist}

In this subsection, we consider the case $\nu>\nu^*$ and $\mathbf E(L\log^+L)<\infty$. By \eqref{gamma_range} and Lemma \ref{lemma_com}, $\frac{\gamma(\lambda)}{\lambda}$ strictly decreases from $+\infty$ to $\nu^*$ on $[0,\lambda^*]$. Therefore, for any $\nu>\nu^*$ there exists a unique $\lambda\in(0,\lambda^*)$ such that $\nu=\frac{\gamma(\lambda)}{\lambda}$.
Recall that the additive martingale $W_t(\lambda)$ is defined in \eqref{mart_add}. Using the periodicity of  $g(\cdot)$ and $\psi(\cdot,\lambda)$, we get that for $y-x\in\mathbb{Z}$,
\begin{equation}\label{mart_W_id1}
(W(\lambda,y), \P_y) \overset{d}{=}( e^{-\lambda(y-x)} W(\lambda,x), \P_x).
\end{equation}
Note that for any $t>s>0$,
\begin{equation*}
W_t(\lambda) = e^{-\gamma(\lambda)s} \sum_{u\in N_s} e^{-\gamma(\lambda)(t-s)}  \sum_{v\in N_t, v>u} e^{-\lambda X_v(t)} \psi(X_v(t),\lambda).
\end{equation*}
It is easy to see that, under $\P_x$,
\begin{equation}\label{mart_W_dec}
W_t(\lambda) \overset{d}{=} \sum_{u\in N_s} e^{-\gamma(\lambda)s} W_{t-s}^{(u)}(\lambda, X_u(s)),
\end{equation}
where $W_{t-s}^{(u)}(\lambda, X_u(s))$ is the additive martingale of the BBMPE
starting from $X_u(s)$, and given $\F_s$, $\{W_{t-s}^{(u)}(\lambda, X_u(s)), u\in N_s\}$ are independent. Hence, letting $t\rightarrow\infty$, we have under $\P_x$
\begin{equation}\label{mart_W_id2}
W(\lambda,x) \overset{d}{=} e^{-\gamma(\lambda)s} \sum_{u\in N_s} W^{(u)}(\lambda, X_u(s)),
\end{equation}
where $W^{(u)}(\lambda, X_u(s))$ is the limit of the additive martingale for the
BBMPE starting from $X_u(s)$, and given $\F_s$, $\{W^{(u)}(\lambda, X_u(s)): u\in N_s\}$ are independent.
\bigskip

\begin{thrm}\label{thrm_exist_super}
	Suppose $|\nu|>\nu^*$ and $\mathbf E(L\log^+L) < \infty$.
	Define
	\begin{equation}\label{def-u-supercritical}
	\mathbf u(t,x) := \E_x \exp\left\{-e^{\gamma(\lambda)t} W(\lambda,x) \right\},
	\end{equation}
	where $|\lambda| \in (0,\lambda^*)$ is such that $\nu = \frac{\gamma(\lambda)}{\lambda}$.
	Then $\mathbf u$ is a  pulsating travelling wave with speed $\nu$.
\end{thrm}

\begin{proof}
We assume that $\lambda\geq 0$.  The case $\lambda < 0$ can be analyzed by symmetry.
By \eqref{mart_W_id2} and the Markov property,
we have, for any $t\geq s\geq 0$,
\begin{align*}
\mathbf u(t,x) &= \E_x \exp\left\{-e^{\gamma(\lambda)t} W(\lambda,x) \right\}
= \E_x \exp\bigg{\{}-e^{\gamma(\lambda)t} e^{-\gamma(\lambda)s} \sum_{u\in N_s} W^{(u)}(\lambda, X_u(s)) \bigg{\}}\\
&= \E_x \bigg{(} \E_x \bigg{[} \exp\big{\{}-e^{\gamma(\lambda)(t-s)}  \sum_{u\in N_s} W^{(u)}(\lambda, X_u(s)) \big{\}}   \;\bigg{|}\; \F_s   \bigg{]} \bigg{)}\\
&= \E_x \prod_{u\in N_s} \E_{X_u(s)} \exp\left\{-e^{\gamma(\lambda)(t-s)} W^{(u)}(\lambda, X_u(s)) \right\}\\
&= \E_x \prod_{u\in N_s} \mathbf u(t-s, X_u(s)).
\end{align*}
In particular, setting $s=t$, we have
\begin{equation}\label{u(0)}
\mathbf u(t,x) = \E_x \prod_{u\in N_t} \mathbf u(0, X_u(t)).
\end{equation}
An argument similar to the one used in \cite{Mc} shows that $\mathbf u(t,x)$ satisfies
\begin{equation*}
\frac{\partial \mathbf u}{\partial t} = \frac{1}{2} \frac{\partial^2 \mathbf u}{\partial x^2} +
\mathbf g\cdot(\mathbf f(\mathbf u)-\mathbf u).
\end{equation*}
Moreover, since  $\frac{\gamma(\lambda)}{\lambda}=\nu$, by \eqref{mart_W_id1},
we have
\begin{align*}
\mathbf u(t+\frac{1}{\nu},x) &= \E_x \exp\left\{-e^{\gamma(\lambda)(t+\frac{1}{\nu})} W(\lambda,x) \right\} = \E_x \exp\left\{-e^{\gamma(\lambda)t+\lambda} W(\lambda,x) \right\}\\
&= \E_{x-1} \exp\left\{-e^{\gamma(\lambda)t} W(\lambda,x-1) \right\} = \mathbf u(t,x-1).
\end{align*}

In order to prove that $\mathbf u(t,x)$ is a pulsating travelling wave, it remains to show that
\begin{equation}
\lim_{x\rightarrow-\infty} \mathbf u(t,x) = 0, \quad \lim_{x\rightarrow +\infty} \mathbf u(t,x) = 1.
\end{equation}
Let $\lfloor x \rfloor$ denote the integral part of $x$. By \eqref{mart_W_id1},
\begin{equation}
\lim_{x\rightarrow +\infty} \mathbf u(t,x) = \lim_{x\to\infty}\E_{x-\lfloor x \rfloor} \exp\left\{-e^{\gamma(\lambda)t} e^{-\lambda \lfloor x \rfloor} W(\lambda,x-\lfloor x \rfloor) \right\}.
\end{equation}
Since $\lim_{x\rightarrow +\infty} e^{-\lambda \lfloor x \rfloor} = 0$ and $y = x - \lfloor x \rfloor \in [0,1)$, we have
\begin{equation}
e^{-\lambda n} W(\lambda,y)\to 0,\quad \mbox{ $\P_y$-almost surely as } n\to\infty.
\end{equation}
It follows from the bounded dominated convergence theorem that for fixed $t\geq 0$,
\begin{equation}
\lim_{n\rightarrow\infty} \mathbf u(t,y+n) =  \E_y \lim_{n\rightarrow\infty} \exp\left\{-e^{\gamma(\lambda)t} e^{-\lambda n} W(\lambda,y) \right\} = 1.
\end{equation}
By Dini's theorem, we have
\begin{equation}
\lim_{x\rightarrow +\infty} \mathbf u(t,x) = 1.
\end{equation}
To prove $\lim\limits_{x\rightarrow-\infty} \mathbf u(t,x) = 0$, we need to verify that
\begin{equation}\label{e:new}
\P_x (W(\lambda,x) = 0) = 0,\quad x\in \R.
\end{equation}
Thanks to \eqref{mart_W_id1} and \eqref{mart_W_id2}, an argument similar to the one used in the proof of Proposition \ref{lemma_deri_conver} shows that $\P_x (W(\lambda,x) = 0) = \P_y (W(\lambda,y) = 0)$ for any $x,y\in\R$. Let $q$ denote this common value.
It follows from \eqref{mart_W_id2} that $q$ must satisfy $q=\E_x q^{|N_s|}$ for any $s>0$. Note that $|N(s)|\geq 1$ almost surely and $\P_x(|N(s)|>1)>0$, thus $q = \E_x q^{|N_s|}$
implies $q=0$ or $1$.
By Theorem \ref{thrm1}, $W(\lambda,x)$ is an $L^1(\P_x)$-limit and hence $q<1$. So $q=0$. Therefore $\mathbf u$ defined by \eqref{def-u-supercritical} is a pulsating travelling wave.
\end{proof}

\subsection{Existence in the critical case ($\nu=\nu^*$)}

Recall that we have under $\P_y$,
\begin{equation}\label{deriv_decomp}
\partial W(\lambda,y) \overset{d}{=} e^{-\gamma(\lambda)s} \sum_{u\in N_s} \partial W^{(u)}(\lambda,X_u(s)).
\end{equation}
An argument similar to the one used in Section \ref{section_exist} leads to the following result.

\begin{thrm}\label{thrm_exist_critical}
Suppose $|\nu|=\nu^*$ and $\mathbf E(L(\log^+L)^2)<\infty$.
Define
\begin{equation*}
\mathbf u(t,x) :=
\E_x\left(\exp\left\{ -e^{\gamma(\lambda^*)t} \partial W(\lambda^*,x) \right\} \right).
\end{equation*}
Then $\mathbf u$ is a pulsating travelling wave with speed $\nu^*$,
and
$$\mathbf u(t,x) = \E_x\left(\exp\left\{ -e^{\gamma(\lambda^*)t} \partial W(-\lambda^*,x) \right\} \right)$$
is a pulsating travelling wave with speed $-\nu^*$.
\end{thrm}

\begin{proof}
We assume that $\lambda\geq 0$.  The case $\lambda < 0$ can be analyzed by symmetry.
The proof of $\mathbf u(t,x)$ being a pulsating travelling wave
is similar to the proof of Theorem \ref{thrm_exist_super}. The decomposition \eqref{decomp_deriv} implies $\mathbf u(t,x)$ satisfies the F-KPP equation \eqref{KPPeq2}. For the derivative martingale, we also have for $y-x\in\mathbb{Z}$,
\begin{equation}
(\partial W(\lambda^*,y), \P_y) \overset{d}{=} (e^{-\lambda^*(y-x)}\partial W(\lambda^*,x), \P_x).
\end{equation}
It follows that
\begin{equation*}
\mathbf u\left(t+\frac{1}{\nu^*},x\right) = \mathbf u(t,x-1).
\end{equation*}
By Theorem \ref{thrm2}, we obtain $\P_y(\partial W(\lambda^*,y)=0) = 0$ when $\mathbf E(L(\log^+L)^2)<\infty$. Therefore,
\begin{equation}
\lim_{x\rightarrow-\infty} \mathbf u(t,x) = 0, \quad \lim_{x\rightarrow +\infty} \mathbf u(t,x) = 1.
\end{equation}
So $\mathbf u(t,x)$ is a pulsating travelling wave with speed $\nu^*$.
\end{proof}

\subsection{Proof of Theorem \ref{thrm3}}

The following result about the extremes of BBMPE is a consequence  of Theorem \ref{thrm1} and \eqref{e:new}.

\begin{lemma}\label{lemma_Mt}
 Let $\widetilde m_t := \max\{X_u(t): u\in N_t \}$ and $m_t:= \min\{X_u(t): u\in N_t \}$. If $\mathbf E(L\log^+L)<\infty$, then for any $x\in \R$,
\begin{equation}\label{Mt}
\lim_{t\uparrow\infty} \frac{\widetilde m_t}{t} = \nu^*
\quad \text{and} \quad \lim_{t\uparrow\infty}
\frac{m_t}{t} = -\nu^*,
\quad \P_x\mbox{-a.s.}
\end{equation}
\end{lemma}
\begin{proof}
We first show that
\begin{equation}
\limsup_{t\uparrow\infty} \frac{-m_t}{t} \leq \nu^*.
\end{equation}
If this were not true, there would exist $\lambda>\lambda^*$ such that
\begin{equation}
\limsup_{t\uparrow\infty} \frac{-m_t}{t} > \frac{\gamma(\lambda)}{\lambda} > \nu^*.
\end{equation}
Hence,
\begin{equation}
W_t(\lambda) \geq
e^{-\lambda m_t-\gamma(\lambda)t} \psi(m_t,\lambda)
= e^{\lambda t\left(\frac{-m_t}{t}-\frac{\gamma(\lambda)}{\lambda}\right)} \psi(m_t,\lambda).
\end{equation}
Then we have
\begin{equation*}
\limsup_{t\uparrow\infty} W_t(\lambda) \geq \limsup_{t\uparrow\infty}
e^{\lambda t \left(\frac{-m_t}{t} - \frac{\gamma(\lambda)}{\lambda}\right)} \psi(m_t,\lambda) = +\infty,
\end{equation*}
which contradicts Theorem \ref{thrm1}.

Next we show that
\begin{equation}
\liminf_{t\uparrow\infty} \frac{-m_t}{t} \geq \nu^*.
\end{equation}
For any small $\delta, \epsilon>0$,
let $\lambda = \lambda^*-\delta$. By the mean value theorem, there
exists $\tilde\lambda\in (\lambda-\epsilon,\lambda)$ with
\begin{equation}\label{mean_value}
\gamma(\lambda) - \gamma(\lambda-\epsilon) = \gamma'(\tilde\lambda) \epsilon.
\end{equation}
For any fixed $\lambda$ and $\lambda-\epsilon$,
there exist $C_1,C_2>0$ such that
$C_1\leq \psi(x,\lambda), \psi(x,\lambda-\epsilon) \leq C_2$ for any $x\in\R$.
Using an argument similar to that of \cite[Corollary 3.2]{KLMR}, we get that
\begin{align*}
\limsup_{t\uparrow\infty}\; &e^{-\gamma(\lambda)t} \sum_{u\in N_t} e^{-\lambda X_u(t)} \psi(X_u(t),\lambda) \textbf{1}_{\{X_u(t) \geq (-\gamma'(\tilde\lambda)+\epsilon)t  \}}\\
\leq &\limsup_{t\uparrow\infty} C_2 e^{-\gamma(\lambda)t} \sum_{u\in N_t} e^{-(\lambda-\epsilon) X_u(t)} e^{-\epsilon X_u(t)} \textbf{1}_{\{X_u(t) \geq (-\gamma'(\tilde\lambda)+\epsilon)t  \}}\\
\leq &\limsup_{t\uparrow\infty} C_2 e^{-\gamma(\lambda)t} \sum_{u\in N_t} e^{-(\lambda-\epsilon) X_u(t)} e^{-\epsilon (-\gamma'(\tilde\lambda)+\epsilon)t }\\
= &\limsup_{t\uparrow\infty} C_2 e^{-(\gamma(\lambda)-\gamma'(\tilde\lambda)\epsilon)t} \sum_{u\in N_t} e^{-(\lambda-\epsilon) X_u(t)} e^{-\epsilon^2 t}\\
\leq & \limsup_{t\uparrow\infty} \frac{C_2}{C_1}  e^{-\gamma(\lambda-\epsilon)t} \sum_{u\in N_t} e^{-(\lambda-\epsilon) X_u(t)} \psi(X_u(t),\lambda-\epsilon) e^{-\epsilon^2 t}\\
= & \limsup_{t\uparrow\infty} \frac{C_2}{C_1} e^{-\epsilon^2 t} W_t(\lambda-\epsilon) = 0,
\end{align*}
where in the last inequality we used  \eqref{mean_value}.
Therefore,
\begin{equation}
\lim_{t\uparrow\infty} e^{-\gamma(\lambda)t} \sum_{u\in N_t} e^{-\lambda X_u(t)} \psi(X_u(t),\lambda) \textbf{1}_{\{X_u(t) < (-\gamma'(\tilde\lambda)+\epsilon)t  \}} = W(\lambda, x),\quad \P_x\mbox{-a.s.}
\end{equation}
By \eqref{e:new}, $\P_x(W(\lambda,x)=0) = 0$.
Thus the previous limit implies that
\begin{equation}
\liminf_{t\uparrow\infty} \mathbf{1}_{\{\exists \; u\in N_t: X_u(t) < (-\gamma'(\tilde\lambda)+\epsilon)t\}} > 0.
\end{equation}
This yields
\begin{equation}
\liminf_{t\uparrow\infty}
\frac{-m_t}{t} \geq \gamma'(\tilde\lambda) - \epsilon.
\end{equation}
Since $\epsilon,\delta$ are arbitrary and $\gamma'$ is continuous, we obtain
$\liminf\limits_{t\uparrow\infty}
 \frac{-m_t}{t} \geq \gamma'(\lambda^*) = \nu^*$.
Thus
\begin{equation}
\lim_{t\uparrow\infty} \frac{m_t}{t} = -\nu^*.
\end{equation}
Using the evenness of $\gamma(\lambda)$ and an argument similar as above, we can easily get that
$\lim_{t\uparrow\infty} \frac{\widetilde m_t}{t} = \nu^*$.
\end{proof}

\begin{proof}[Proof of Theorem \ref{thrm3}]

(i) follows from Theorems \ref{thrm_exist_super}.
(ii) follows from  Theorems \ref{thrm_exist_critical}.
Now we prove (iii). If the conclusion were false,  let $\mathbf u(t,x)$ denote the pulsating travelling wave with speed $\nu<\nu^*$.
By the uniqueness of solutions of initial value problem
\begin{align*}
\frac{\partial \mathbf u}{\partial t} = \frac{1}{2} \frac{\partial^2 \mathbf u}{\partial x^2} +
\mathbf g\cdot(\mathbf f(\mathbf u)-\mathbf u),
\end{align*}
with initial value $\mathbf u(0,x)$, we have
\begin{equation}
\mathbf u(t,x) = \E_x \prod_{u\in N_t} \mathbf u(0,X_u(t)).
\end{equation}
Noting that $\mathbf u(t+\frac{1}{\nu},x) = \mathbf u(t,x-1)$, we get that for $\nu t\in \N$,
\begin{align}
\mathbf u(0,x) = \mathbf u(t, x+\nu t) = \E_{x+\nu t} \prod_{u\in N_t} \mathbf u(0,X_u(t)) = \E_{x} \prod_{u\in N_t} \mathbf u(0,X_u(t)+\nu t),
\end{align}
where the last equality follows from the periodicity. Since $0\leq \mathbf u(t,x)\leq 1$, by the dominated convergence theorem, we have
\begin{align*}
\mathbf u(0,x) &= \lim_{t\rightarrow\infty, \nu t\in\N} \E_{x} \prod_{u\in N_t} \mathbf u(0,X_u(t)+\nu t) = \E_{x} \lim_{t\rightarrow\infty, \nu t\in\N} \prod_{u\in N_t} \mathbf u(0,X_u(t)+\nu t)\\ &\leq \E_x \lim_{t\rightarrow\infty, \nu t\in\N} \mathbf u(0, m_t + \nu t) = 0,
\end{align*}
here we used $\lim_{t\uparrow\infty} (m_t + \nu t) = -\infty$ $\P_x$-almost surely,
which follows from Lemma \ref{lemma_Mt}.
This leads to a contradiction.
\end{proof}

\section{Appendix}
Recall that, for any $c>0$, $A_c$ is defined in \eqref{def-Ac}.
\begin{lemma}\label{Prob-Ac}
For any $c>0$,
$\widetilde{\P}_y^{(x,\lambda)}(A_c) = 0$,
\end{lemma}

\begin{proof}  The proof is almost the same as the proof \cite[(8)]{YR}.
The only changes are some notation and fixing of a few typos.
Note that under $\widetilde{\P}_y^{(x,\lambda)}$, $x - \frac{\phi_{\lambda}(y,\lambda^*)}{\phi(y,\lambda^*)} + R_t$ is a Bessel-3 process starting from $x- \frac{\phi_{\lambda}(y,\lambda^*)}{\phi(y,\lambda^*)}$.
For simplicity, we still use $\tilde x$
to denote $x- \frac{\phi_{\lambda}(y,\lambda^*)}{\phi(y,\lambda^*)}$. Then $\tilde x+R_t$ has the same law as the modulus process of $W_{t}+\hat{x}$, where
    $\left\{W_{t}, t\geq 0; \mathbf{P}_w\right\}$
is a three-dimensional standard Brownian motion and $\hat{x}$ is a point in $\mathbb{R}^{3}$ with norm $\tilde x$. We still use $A_{c}$ to denote the same set corresponding to $\left\{W_{t}, t\geq 0;\mathbf{P}_w\right\}$.

\begin{align}\label{YR_9}
c &\geq \widetilde{\P}_y^{(x,\lambda)} \left( \mathbf{1}_{A_c} \int_0^{+\infty} (m+1) \sum_{k\in\mathbb{Z}_+} \tilde{p}_k \mathbf{1}_{\left\{k(\tilde x + R_t) e^{-\lambda^* R_t} \geq M\right\} } \mathrm{d}t \right)\notag \\
&=\int_0^{+\infty} (m+1) \sum_{k\in\mathbb{Z}_+} \tilde{p}_k \widetilde{\P}_y^{(x,\lambda)}(\mathbf{1}_{A_c} \mathbf{1}_{\left\{(\tilde x + R_t) e^{-\lambda^*(\tilde x+R_t) } \geq Mk^{-1}e^{-\lambda^*\tilde x} \right\} } )   \mathrm{d}t \notag \\
&= (m+1) \sum_{k\in\mathbb{Z}_+} \tilde{p}_k \int_0^{+\infty} \mathbf{P}_w\left(\mathbf{1}_{A_c} \mathbf{1}_{\left\{ |W_t+\hat{x}|e^{-\lambda^*|W(t)+\hat{x}|} \geq Mk^{-1}e^{-\lambda^*\tilde x} \right\} } \right)  \mathrm{d}t.
\end{align}

We claim that there exists $K_{1}>1$ such that when $k \geq K_{1}$
\begin{equation}\label{YR_10}
\left\{y \in \mathbb{R}^{3}: 1+\tilde x \leq|y| \leq \frac{\log k}{ \lambda^*}\right\} \subset\left\{y \in \mathbb{R}^{3}:|y+\hat{x}| e^{-\lambda^*|y+\hat{x}|} \geq M k^{-1} e^{-\lambda^* \tilde x}\right\} .
\end{equation}
In fact, $1+\tilde x \leq|y| \leq \frac{\log k}{ \lambda^*}$ implies $1 \leq|y+\hat{x}| \leq \frac{\log k}{\lambda^*}+\tilde x.$ Consider the function $f(x)=\tilde x e^{-\lambda^*\tilde x} .$ On the positive half line, it
increases to a supremum and then decreases to 0 as $x$ goes to infinity. Thus we can find $K_{1}>1$ large enough such that when $k \geq K_{1}$,
\begin{align*}
1+\tilde x \leq|y| \leq \frac{\log k}{\lambda^*} & \Rightarrow f(|y+\hat{x}|) \geq f\left(\frac{\log k}{ \lambda^*}+\tilde x\right) \\
& \Rightarrow|y+\hat{x}| e^{-\lambda^*|y+\hat{x}|} \geq\left(\frac{\log k}{ \lambda^*}+\tilde x\right) k^{-1} e^{-\lambda^*\tilde x}.
\end{align*}
Thus \eqref{YR_10} is valid.

We continue the estimate \eqref{YR_9} when $k \geq K_{1}$,
\begin{align}\label{YR_11}
c & \geq (m+1) \sum_{k: k \geq K_{1}} \tilde{p}_{k} \int_{0}^{+\infty} \mathbf{P}_w \left(\mathbf{1}_{A_{c}} \mathbf{1}_{\left\{1+\tilde x \leq\left|W_{t}\right| \leq \frac{\log k}{\lambda^*}\right\}}\right) \mathrm{d} t \notag \\
&=(m+1) \sum_{k: k \geq K_{1}} \tilde{p}_{k} \mathbf{P}_w \left(\mathbf{1}_{A_{c}} \int_{0}^{+\infty} \mathbf{1}_{\left\{1+\tilde x \leq\left|W_{t}\right| \leq \frac{\log k}{\lambda^*}\right\}} \mathrm{d} t\right).
\end{align}
$\left(\left|W_{t}\right|, t\geq 0;\mathbf{P}_w\right)$ is a Bessel-3 process starting from 0. Let $\left\{l^{a}: a \geq 0\right\}$ be the family of its local times, then the process $\left\{l_{\infty}^{a}, a \geq 0\right\}$ is a BESQ$^{2}(0)$ process which implies $l_{\infty}^{a} \overset{d}{=} a l_{\infty}^{1}$ and $\mathbf{P}_w\left(l_{\infty}^{1}=0\right)=0$ (see Revuz and Yor \cite{RY}, p. 425, Ex. 2.5). Thus
\begin{align}\label{YR_12}
    \mathbf{P}_w \left(\mathbf{1}_{A_{c}} \int_{0}^{+\infty} 1_{\left\{1+\tilde x \leq\left|W_{t}\right| \leq \frac{\log k}{\lambda^*}\right\}} \mathrm{d} t\right) &=\mathbf{P}_w \left(1_{A_{c}} \int_{1+\tilde x}^{\frac{\log k}{\lambda^*}} l_{\infty}^{a} \mathrm{d} a\right) \notag \\
	&=\mathbf{P}_w \left(\mathbf{1}_{A_{c}} \int_{1+\tilde x}^{\frac{\log k}{\lambda^*}} a \mathrm{d} a \int_{0}^{a^{-1} l_{\infty}^{a}} \mathrm{d} u\right) \notag \\
	&=\int_{1+\tilde x}^{\frac{\log k}{\lambda^*}} a \mathrm{d} a \int_{0}^{+\infty} \mathbf{P}_w \left(\mathbf{1}_{A_{c}} \mathbf{1}_{\left\{u \leq a^{-1} l_{\infty}^{a}\right\}}\right) \mathrm{d} u.
\end{align}
Note that
\begin{equation*}
\mathbf{P}_w\left(\mathbf{1}_{A_{c}} \mathbf{1}_{\left\{u \leq a^{-1}l_{\infty} ^{a}\right\}}\right) \geq \left(\mathbf{P}_w\left(A_{c}\right) - \mathbf{P}_w\left(a^{-1} l_{\infty}^{a}<u\right)\right)^{+} = \left(\mathbf{P}_w\left(A_{c}\right) - \mathbf{P}_w\left(l_{\infty}^{1}<u\right)\right)^{+},
\end{equation*}
and there exist  $C>0$ and $K_{2}>1$ such that for  $k \geq K_{2}$
\begin{equation*}
\int_{1+\tilde x}^{\frac{\log k}{\lambda^*}} a \mathrm{d} a=\frac{1}{2}\left(\left(\frac{\log k}{ \lambda^*}\right)^{2}-(1+\tilde x)^{2}\right) \geq C(\log k)^{2}.
\end{equation*}
Then \eqref{YR_12} implies
\begin{equation}\label{YR_13}
\mathbf{P}_w \left(\mathbf{1}_{A_{c}} \int_{0}^{+\infty} \mathbf{1}_{\left\{1+\tilde x \leq\left|W_{t}\right| \leq \frac{\log k}{\lambda^*}\right\}} \mathrm{d} t\right) \geq C(\log k)^{2} \int_{0}^{\infty}\left(\mathbf{P}_w \left(A_{c}\right)-\mathbf{P}_w \left(l_{\infty}^{1}<u\right)\right)^{+} \mathrm{d}u.
\end{equation}
Set $K=K_{1} \vee K_{2}.$ Using \eqref{YR_11} and \eqref{YR_13} we get
\begin{equation}\label{YR_14}
\sum_{k: k \geq K} \tilde{p}_{k}(\log k)^{2} \int_{0}^{+\infty} \left(\mathbf{P}_w \left(A_{c}\right) - \mathbf{P}_w\left(l_{\infty}^{1}<u\right)\right)^{+} \mathrm{d} u<+\infty.
\end{equation}
The assumption $\mathbf E (L\left(\log ^{+} L\right)^{2})=+\infty$ is equivalent to $\sum_{k \in \mathbb{Z}_{+}} \tilde{p}_{k}\left(\log ^{+} k\right)^{2}=+\infty$. Then by \eqref{YR_14},
\begin{equation*}
\int_{0}^{+\infty}\left(\mathbf{P}_w\left(A_{c}\right) - \mathbf{P}_w\left(l_{\infty}^{1}<u\right)\right)^{+} \mathrm{d} u=0.
\end{equation*}
Thus $\mathbf{P}_w(A_{c})=0$ since $\mathbf{P}_w\left(l_{\infty}^{1}=0\right)=0$,
consequently $\widetilde{\P}_y^{(x,\lambda)}(A_c) = 0$.
\end{proof}



\begin{thebibliography}{99}



\bibitem{BK04}  J. D. Biggins and A. E. Kyprianou. Measure change in multitype branching.
\emph{Adv.
in Appl. Probab.} {\bf36} (2004) 544-581.

\bibitem{Bramson78} M. Bramson. Maximal displacement of branching Brownian motion. \emph{Common. Pure Appl. Math.} {\bf31} (1978) 531-581.

\bibitem{Bramson83} M. Bramson. Convergence of solutions to the Kolmogorov equation to travelling waves. \emph{Mem. Amer. Math. Soc.} {\bf44} (1983) iv+190 pp.


\bibitem{BH}  H. Berestycki and F. Hamel. Front propagation in periodic excitable media. \emph{Comm. Pure Appl.
Math.} {\bf 55} (2002) 949-1032.

\bibitem{CR88} B. Chauvin and A. Rouault. KPP equation and supercritical branching Brownian motion in the subcritical speed area. Application to spatial trees. \emph{Probab. Theory Related Fields} {\bf80} (1988) 299-314.


\bibitem{Durrett} R. Durrett. \emph{Probability: Theory and Examples.} Cambridge University Press,
 Cambridge,
fourth edition, 2010.

\bibitem{DZ} A. Dembo and O. Zeitouni. \emph{Large deviations techniques and applications, volume 38 of Stochastic Modelling and Applied Probability.} Springer-Verlag, Berlin, 2010.


\bibitem{Dynkin} E. B. Dynkin. Path processes and historical superprocesses. \emph{Probab. Theory Related Fields} {\bf 90} (1991), 1-36.


\bibitem{Fish}  R. A. Fisher. The wave of advance of advantageous genes. \emph{Ann. Eugenics} {\bf7} (1937)
355-369.

\bibitem{Harris99} S. C. Harris. Travelling waves for the F-K-P-P equation via probabilistic arguments. \emph{Proc. Roy. Soc. Edinburgh Sect. A} {\bf129} (1999) 503-517.	

\bibitem{HH09}  R. Hardy and S. C. Harris. A spine approach to branching diffusions with applications to $L^p$-convergence of martingales.
S\'{e}minaire de Probabilit\'{e}s
{\bf XLII} (2009) 281-330.

\bibitem{HR09} S. C. Harris and M. I. Roberts. Measure changes with extinction. \emph{Stat. Probab. Lett.}
{\bf 79} (2009) 1129-1133.

\bibitem{HR17} S. C. Harris and M. I. Roberts. The many-to-few lemma and multiple spines. \emph{Ann. Inst. Henri Poincar\'{e} Probab. Stat.}
{\bf53} (2017) 226-242.

\bibitem{H08}
F. Hamel. Qualitative properties of monostable pulsating fronts: exponential decay and monotonicity.
\emph{J. Math. Pures Appl.} {\bf 89} (2008) 355-399.

\bibitem{HR} F. Hamel, and L. Roques. Uniqueness and stability properties of monostable pulsating fronts.
\emph{J. Eur. Math. Soc.} {\bf 13} (2011) 345-390.

\bibitem{HNRR} F. Hamel, J. Nolen, J.-M. Roquejoffre and L. Ryzhik. The logarithmic delay of KPP fronts in a periodic medium. \emph{J. Eur. Math. Soc. (JEMS)} {\bf18} (2016) 465-505.

\bibitem{Imhof} J.-P. Imhof. Density factorizations for Brownian motion,
meander and the three-dimensional Bessel process, and
applications. \emph{J. Appl. Probab.} {\bf 21} (1984) 500-510.

\bibitem{Jagers} P. Jagers. General branching processes as Markov fields. \emph{Stochastic Process.
Appl.}   {\bf 32} (1989) 183-212.

\bibitem{KPP} A. Kolmogorov, I. Petrovskii and N. Piskounov. \'{E}tude de I'\'{e}quation de la diffusion avec croissance de la quantit\'{e} de la mati\`{e}re at son application a un probl\`{e}m biologique. \emph{Moscow Univ.Math. Bull.} {\bf1} (1937) 1-25.


\bibitem{Ky} A. E. Kyprianou. Travelling wave solution to the K-P-P equation: Alternatives to Simon Harris' probabilistic analysis. \emph{Ann. Inst. Henri Poincar\'{e} Probab. Stat.}
{\bf40} (2004) 53-72.

\bibitem{KLMR} A. E. Kyprianou, R.-L. Liu, A. Murillo-Salas and Y.-X. Ren. Supercritical super-Brownian motion with a general branching mechanism and travelling waves. \emph{Ann. Inst. Henri Poincar\'{e} Probab. Stat.} {\bf48} (2012) 661-687.

\bibitem{LTZ} E. Lubetzky, C. Thornett, and O. Zeitouni.
Maximum of branching Brownian motion in a periodic environment.
\emph{Ann. Inst. Henri Poincar\'{e} Probab. Stat.}, to appear.

\bibitem{LPP95} R. Lyons, R. Pemantle, and Y. Peres. Conceptual proofs of $L\log L$ criteria for mean behavior of branching processes. \emph{Ann. Probab.} {\bf 23} (1995) 1125-1138.


\bibitem{Maillard} P. Maillard. Branching Brownian motion with selection. Ph.D. thesis, 2012. Available at arXiv:1210.3500.

\bibitem{Mc} H. P. McKean. Application of Brownian motion to the equation of Kolmogorov-Petrovskii-Piskunov. \emph{Comm. Pure Appl. Math.} {\bf28} (1975) 323-331.


\bibitem{RS} Y.-X. Ren and R. Song. Spine decomposition for branching Markov processes
and its applications. 2020. Available at arXiv:2007.12495.


\bibitem{RSYb} Y.-X. Ren, R. Song, and F. Yang. Branching Brownian motion in a periodic environment and
uniqueness of pulsating travelling waves.


\bibitem{RY} D. Revuz and M. Yor. \emph{Continuous Martingales and Brownian Motion.} Springer-Verlag,
Berlin,
1991.

\bibitem{YR} T. Yang and Y.-X. Ren. Limit theorem for derivative martingale at criticality w.r.t. branching Brownian motion. \emph{Statist. Probab. Lett.} {\bf 81} (2011) 195-200.


\end{thebibliography}
\end{document}